\documentclass{amsart}

\usepackage{mathrsfs,amsmath,amssymb,amsthm,bbm,cite}
\usepackage{hyperref}

\newtheorem{lem}{Lemma}[section]
\newtheorem{thrm}[lem]{Theorem}
\newtheorem{prop}[lem]{Proposition}
\newtheorem{cor}[lem]{Corollary}
\theoremstyle{definition}
\newtheorem{defn}[lem]{Definition}
\newtheorem{rem}[lem]{Remark}

\renewcommand{\Re}{\ensuremath{\operatorname{Re}}}
\renewcommand{\Im}{\ensuremath{\operatorname{Im}}}
\renewcommand{\epsilon}{\varepsilon}
\newcommand{\mc}{\mathcal}
\newcommand{\mb}{\mathbf}
\newcommand{\mr}{\mathrm}
\newcommand{\mf}{\mathfrak}

\newcommand{\mtt}{\mathtt}
\numberwithin{equation}{section}
\allowdisplaybreaks

\newcommand{\R}{\ensuremath{\mathbb{R}}}
\newcommand{\C}{\ensuremath{\mathbb{C}}}

\newcommand{\Z}{\ensuremath{\mathbb{Z}}}

\newcommand{\Schwartz}{\ensuremath{\mathcal S(\R)}}
\newcommand{\Test}{\ensuremath{C^\infty_c}}

\DeclareMathOperator{\supp}{supp}

\newcommand{\<}{\ensuremath{\langle}}
\renewcommand{\>}{\ensuremath{\rangle}}

\newcommand{\p}{\ensuremath{\partial}}

\newcommand{\eps}{\varepsilon}

\newcommand{\qtq}[1]{\quad\text{#1}\quad}
\newcommand{\LHS}[1]{{\rm{LHS}}\eqref{#1}}
\newcommand{\RHS}[1]{{\rm{RHS}}\eqref{#1}}

\newcommand{\I}{\mf I}

\newcommand{\op}{\mr{op}}
\DeclareMathOperator{\tr}{tr}

\newcommand{\sbrack}[1]{^{(#1)}}
\newcommand{\ta}{\mtt a}
\newcommand{\tb}{\mtt b}

\newcommand{\vk}{\varkappa}
\newcommand{\bB}{\mb B}
\newcommand{\bd}{\mb d}
\newcommand{\bG}{\mb G}
\newcommand{\bS}{\mb S}
\newcommand{\cB}{\mc B}
\usepackage{tikz}
\newcommand{\zero}{z}
\newcommand{\Id}{\text{Id}}

\title{Large-data equicontinuity for the derivative NLS}

\author[B.~Harrop-Griffiths]{Benjamin Harrop-Griffiths}
\address{Benjamin Harrop-Griffiths\\
Department of Mathematics\\
University of California, Los Angeles, CA 90095, USA}
\email{harropgriffiths@math.ucla.edu}

\author[R.~Killip]{Rowan Killip}
\address{Rowan Killip\\
Department of Mathematics\\
University of California, Los Angeles, CA 90095, USA}
\email{killip@math.ucla.edu}

\author[M.~Vi\c san]{Monica Vi\c{s}an}
\address{Monica Vi\c{s}an\\
Department of Mathematics\\
University of California, Los Angeles, CA 90095, USA}
\email{visan@math.ucla.edu}

\begin{document}

\begin{abstract}
We consider the derivative NLS equation in one spatial dimension, which is known to be completely integrable.  We prove that the orbits of $L^2$ bounded and equicontinuous sets of initial data remain bounded and equicontinuous, not only under this flow, but under the entire hierarchy. This allows us to remove the small-data restriction from prior conservation laws and global well-posedness results. 
\end{abstract}

\maketitle

\section{Introduction}

The derivative nonlinear Schr\"odinger equation,
\begin{equation}\label{DNLS}\tag{DNLS}
i\tfrac d{dt}q + q'' + i\bigl(|q|^2q)' = 0,
\end{equation}
describes the time evolution of a complex-valued function $q$ on the line. (Here and below, primes denote spatial derivatives.)  It arises both as an effective equation in plasma physics \cite{IchikawaWatanabe, MOMT, mjolhus_1976} and as an example of a completely integrable PDE \cite{MR464963}.

As a conspicuous dispersive equation, the well-posedness question for \eqref{DNLS} has received considerable attention over the years.  As we shall discuss more fully below, a robust theory of local well-posedness has been known for some time, as has a small-data global theory.  However, it was only very recently that global well-posedness could be proved for large data (even of Schwartz class).  The central bottle-neck in the theory of this equation has been obtaining satisfactory a priori bounds for solutions.  This is startling --- as a completely integrable system, \eqref{DNLS} admits infinitely many conservation laws!

Preeminent among the conserved quantities for \eqref{DNLS} is 
\begin{align}\label{M_of_q}
M(q) &:= \int |q(x)|^2\,dx. 
\end{align}
This is manifestly coercive; moreover, it is invariant under the scaling
\begin{align}\label{e:scaling}
q(t,x) \mapsto \sqrt{h}\, q(h^2 t,h x) \qtq{for} h>0
\end{align}
that preserves \eqref{DNLS}.  When we spoke of small-data well-posedness, we precisely meant under a restriction on the size of $M(q)$.  

While \eqref{DNLS} admits infinitely many other conserved quantities, such as
\begin{align}
H_1(q)&=-\tfrac{1}{2}\int i (q \bar q'-\bar q q')+|q|^4 \, dx, \\
H_2(q)&=\int |q'|^2 +\tfrac{3}{4} i |q|^2 (q \bar q'-\bar q q') + \tfrac 12 |q|^6\, dx,
\end{align}
it turns out that none are coercive.  We will further justify this claim below when we discuss the forward/inverse scattering technique.  For the moment, let us focus on $M$, $H_1$, and $H_2$.  The failure of their coercivity is best witnessed by a concrete example: the algebraic soliton.  This solution, which will be the central antagonist in our story, has initial data
\begin{align}\label{algae}
q_a(x) = \frac{2(1-ix)}{(1+ix)^2} e^{ix/2} \qtq{and evolution} q(t,x) = q_a(x-t)  e^{it/4} . 
\end{align}
Direct computation shows
$$
M(q_a)=4\pi \qtq{and} H_1(q_a)=H_2(q_a)=0;
$$
moreover, these values are inherited by all rescalings \eqref{e:scaling} of the algebraic soliton.  Thus, no combination of $M$, $H_1$, and $H_2$ can control the $H^1$ norm for all solutions to \eqref{DNLS}.

In fact, the algebraic soliton \eqref{algae} serves as a minimal counter-example to the coercivity of $M$, $H_1$, and $H_2$.  This was proved by Wu in \cite{MR3393674}, who showed that the simultaneous conservation of $M$, $H_1$, and $H_2$ does provide an a priori $H^1$ bound for all solutions with $M(q)< 4\pi$.  

Our goal in this paper is to prove that the flow map for \eqref{DNLS} preserves $L^2$-equicontinuity. This question was posed in \cite{KNV}, where it was shown to have robust consequences both for a priori bounds and for the well-posedness problem. To formulate matters precisely, we need one preliminary definition:

\begin{defn}
A bounded set $Q\subseteq L^2(\R)$ is said to be \emph{$L^2$-equicontinuous} if 
\begin{align*}
\limsup_{y\to 0} \ \sup_{q\in Q}\ \int |q(x-y) - q(x)|^2\,dx = 0.
\end{align*}
\end{defn}

An equivalent formulation of equicontinuity as tightness of the Fourier transforms will be useful later and is presented in \eqref{HF-UNI} below. 

As we do not currently know whether $L^2$ initial data leads to (even local) solutions, we formulate the preservation of equicontinuity for a narrower (yet dense) class of initial data.  Specifically, we will consider Schwartz-class initial data.  As we will discuss below, such initial data leads to global Schwartz solutions.

\begin{thrm}\label{t:DNLS equi}
Given an \(L^2\)-bounded and equicontinuous set \(Q\subseteq \Schwartz\) of Schwartz-class initial data for \eqref{DNLS}, the attendant collection of orbits
$$
\bigl\{q(t) : q(0) \in Q\text{ and }t\in \R\bigr\}
$$
is also \(L^2\)-bounded and equicontinuous.
\end{thrm}

The definition of equicontinuity can easily be adapted to any Banach space of functions on $\R$.  The facts that $L^2$ is scaling-critical for \eqref{DNLS} and that it coincides with the conserved quantity $M$ makes this the most natural (and most ambitious) space to choose here.

Theorem~\ref{t:DNLS equi} is phrased in terms of Schwartz-class solutions, not only because we know that such solutions exist, but also because they provide an effective way to address data with less regularity or decay.

The question addressed in Theorem~\ref{t:DNLS equi} was pinpointed as important for the theory of \eqref{DNLS} in \cite{KNV}; see \cite[Conjecture~1.1]{KNV} as well as the more general \cite[Conjecture~1.2]{KNV} that we will address in Theorem~\ref{t:Main} below.  Theorem~1.3 of~\cite{KNV} settled these questions in the case of sets $Q$ with $\sup\{M(q):q\in Q\} <4\pi$.  Note that $M=4\pi$ is precisely the threshold delineated by the algebraic soliton that we discussed earlier.  Our purpose in this paper is to remove such small data restrictions.

The second major theme of \cite{KNV} was demonstrating important consequences of $L^2$-equicontinuity for \eqref{DNLS}.  In formulating such results, \cite{KNV} introduced a threshold $M_*$ defined by the property that equicontinuity is preserved for sets with $\sup M(q) < M_*$; see \cite[Definition~1.4]{KNV} for the precise formulation.  By proving that $M_*=\infty$ in this paper, we immediately obtain a number of corollaries from \cite{KNV}.  For example, the following a priori bounds for (arbitrarily large) solutions:

\begin{cor}\label{C:s bounds}
Fix $0<s<\frac12$ or $s=1$.  There is a function $C_s:[0,\infty)\to[0,\infty)$ so that
\begin{align}\label{E:s bounds}
\sup_t \| q(t) \|_{H^s(\R)} \leq C_s\bigl( \| q(0) \|_{H^s(\R)} \bigr)
\end{align}
for every solution $q(t)$ to \eqref{DNLS} with initial data $q(0)\in \Schwartz$.
\end{cor}

The possibility of choosing $0<s<\frac12$ or $s=1$ in Corollary~\ref{C:s bounds} reflects the fact that we have grouped together two separate arguments from \cite{KNV}.  For small values of $s$ this result is \cite[Theorem~4.3]{KNV}, which employs conservation of the perturbation determinant; for $s=1$ this recalls \cite[Proposition~4.1]{KNV}, which shows how one may combine equicontinuity with the conserved quantities $M$ and $H_2$ to obtain a priori bounds.

Prior to \cite{KNV}, Klaus and Schippa \cite{klaus2020priori} proved $H^s$-bounds for $0<s<\frac12$ under a (non-quantitative) small-$M$ restriction.  Microscopic conservation laws were derived under the same restriction in \cite{tang2020microscopic}; by the results of this paper, these laws extend to all values of $M$.

As mentioned earlier, the theory of \eqref{DNLS} has been greatly hindered by the lack of large-data a priori bounds in $H^s(\R)$ spaces.  This impasse was only very recently dislodged by Bahouri and Perelman \cite{BP}.  In a major breakthrough, they proved \eqref{E:s bounds} with $s=\frac12$.  Their paper has served as an important source of inspiration for us. 

By employing Corollary~\ref{C:s bounds} as the base step of an inductive scheme, the recent paper \cite{BLP} proved a priori bounds in $H^s(\R)$ spaces for all $s\geq \frac12$.

It follows from Corollary~\ref{C:s bounds} that a satisfactory local well-posedness theory (as has long been known in $H^1(\R)$, for example) is immediately global.  Indeed, local well-posedness guarantees that irregular solutions may be approximated (locally in time) by Schwartz solutions and consequently inherit their bounds.  These bounds allow the local-in-time construction of solutions to be iterated indefinitely.

Prior to \cite{KNV}, equicontinuity was shown to play an important role in proving optimal well-posedness for several other completely integrable dispersive equations, such as KdV, NLS, and mKdV; see \cite{bringmann2019global,harropgriffiths2020sharp,MR3990604}.  For these arguments, however, one needs to know that equicontinuity is preserved not merely for the evolution in question, but for the whole hierarchy.  Just such a conjecture was formulated for \eqref{DNLS} in \cite{KNV} and will be settled in this paper; see Theorem~\ref{t:Main}.  In this way, \cite[Theorem~1.5]{KNV} yields the following:

\begin{cor}\label{C:s GWP}
For $\frac16\leq s<\frac12$, the evolution \eqref{DNLS} is globally well-posed in $H^s(\R)$.
\end{cor}

The notion of well-posedness meant here is this: For any sequence of Schwartz initial data $q_n(0)$ that is convergent in $H^s(\R)$ the corresponding solutions $q_n(t)$ converge in $C_t([-T,T];H^s(\R))$ for every finite $T>0$.  As $s\geq \frac16$, convergence also holds in $L^3([-T,T]\times\R)$; thus, the limiting trajectory is a distributional solution to \eqref{DNLS}.  While these solutions clearly depend continuously on the initial data, it was shown in \cite{MR1837253,MR1693278} that the data-to-solution map cannot be \emph{uniformly} continuous on bounded subsets of $H^s(\R)$ when $s<\frac12$.

The previous state of the art for large data local well-posedness in $H^{s}(\R)$ spaces was $s\geq \frac12$.   Local well-posedness has also been studied in Fourier--Lebesgue spaces, \cite{Deng2, MR2181058, MR2390318}; such spaces are better suited for studying invariant measures and were used for this purpose in \cite{MR2928851}.

Many years before Wu's $4\pi$ result in \cite{MR3393674}, Hayashi and Ozawa \cite{MR1152001} proved that $M$ and $H_2$ control the $H^1$ norm of solutions when $M(q)<2\pi$.  Using this ingredient, they proved that \eqref{DNLS} is globally well-posed in $H^1(\R)$ under this restriction on $M$.  They also proved that such solutions with Schwartz-space initial data remain Schwartz.  

A series of subsequent works, \cite{MR1871414, MR1950826,MR2823664,MR1836810}, culminated in a proof that \eqref{DNLS} is globally well-posed in $H^s(\R)$ for $s\geq\frac12$, in the $M<2\pi$ regime.  Following the discovery of the $4\pi$ threshold in \cite{MR3393674}, the case  $s\geq\frac12$ and $M<4\pi$ was treated in \cite{MR3583477}.

Combining the large-data bounds proved recently by Bahouri and Perelman \cite{BP} with the local theory discussed above shows that \eqref{DNLS} is globally well-posed in $H^s(\R)$ for $s\geq\frac12$.  In this way, \cite{BP} yields the first large-data global well-posedness result in Sobolev spaces.  Corollary~\ref{C:s GWP} improves upon this yielding global well-posedness for $s\geq\frac16$.  This brings us closer to the critical scaling; moreover, even local well-posedness for large initial data was previously unknown for any $s<\frac12$. 

While ill-suited to initial data in $H^s$ spaces (due to the poor physical decay), the inverse scattering technique is extremely powerful, yielding not only well-posedness of integrable equations, but also detailed information on the long-time behavior of solutions.  This approach to \eqref{DNLS} has advanced considerably in recent years through the efforts of many people:  \cite{MR3706093,MR3563476, MR3739932, MR4149070,MR3913998,MR3858827, MR4042219,MR3862117,MR3702542, saalmann2017global}.  Given the formulation of our Theorems~\ref{t:DNLS equi} and~\ref{t:Main}, it is important here to single out the contribution of \cite{MR4149070} which established that \eqref{DNLS} is globally well-posed  in
$H^{2,2}(\R)= \{f\in H^2(\R):\, x^2f\in L^2(\R)\}$.  Combined with the arguments in \cite{MR1152001}, this result shows that Schwartz-class initial data lead to global Schwartz-space solutions.

While it is also interesting to study \eqref{DNLS} when posed on the circle, we are currently unable to prove an analogue of our main theorems in that setting.   The results presented in \cite{KNV} cover both geometries and so a proof of equicontinuity on the circle would have consequences for the well-posedness problem paralleling those described above. 

The integrable nature of \eqref{DNLS} will play a major role in our analysis.  This was first articulated by Kaup and Newell in \cite{MR464963}.  In particular, they introduced operator pencils
\begin{equation}\label{L&P}\begin{aligned}
L(\lambda;q) &= \begin{bmatrix}
   -i \lambda^2 -\partial_x & \lambda q\\
   -\lambda \bar q & i \lambda^2 -\partial_x
    \end{bmatrix}, \\
P(\lambda;q) &= \begin{bmatrix}-2i\lambda^4+i\lambda^2|q|^2 &
				 2\lambda^3q      - \lambda|q|^2q      + i\lambda q' \\[1mm]
				-2\lambda^3\bar q + \lambda|q|^2\bar q + i\lambda \bar q' &
				     2i\lambda^4-i\lambda^2|q|^2  \end{bmatrix},
\end{aligned}
\end{equation}
with spectral parameter $\lambda\in\C$, and proved that
\begin{align}
q(t,x)\text{ solves \eqref{DNLS}} \ \iff\  \partial_t L(\lambda;q(t)) = [P(\lambda;q(t)),L(\lambda;q(t))].  \label{L P rep}
\end{align}

For the particular case $q\equiv 0$, we will write
\begin{align}\label{L0 defn}
L_0(\lambda) &:= \begin{bmatrix}
   -i \lambda^2 -\partial_x & 0\\
   0 & i \lambda^2 -\partial_x
    \end{bmatrix} = -\partial_x - i \lambda^2 \sigma_3 \qtq{where} \sigma_3 = \begin{bmatrix}1&0\\0&-1\end{bmatrix}.
\end{align}

A central object of study in the presence of such a Lax representation \eqref{L P rep} is the scattering matrix.  This connects the asymptotic behavior at the two spatial infinities of solutions to the ODE $L\vec\psi =0$.

Given $\lambda^2\in\R$ and $q\in\Schwartz$, it is known that there is a unique matrix $\Psi$ with 
\begin{align}\label{Psi minus}
L \Psi = 0 \qtq{and} \lim_{x\to-\infty}  \Psi(x) e^{i\lambda^2x\sigma_3} = \Id.
\end{align}
Moreover, the limit
\begin{align}\label{S defn}
S(\lambda) :=\lim_{x\to+\infty} e^{i\lambda^2x\sigma_3} \Psi(x)
\end{align}
exists and satisfies 
\begin{align}\label{S structure}
S(\lambda) = \begin{bmatrix}\;\ta(\lambda)&-\overline{\tb(\bar\lambda)}\;\\[1mm] \;\tb(\lambda)&\overline{\ta(\bar\lambda)}\end{bmatrix}
	\qtq{with}
|\ta(\lambda)|^2 = \begin{cases}1 - |\tb(\lambda)|^2&\text{ if }\lambda\in \R,\\1 + |\tb(\lambda)|^2&\text{ if }\lambda\in i\R.\end{cases}
\end{align}
We will review some of this material in Section~\ref{S:3}.

The key virtue of $\ta$ and $\tb$ is that they evolve in a simple manner as $q(t)$ flows according to \eqref{DNLS}.  Specifically,
\begin{equation}\label{a doesn't}
\ta(\lambda;q(t)) = \ta(\lambda;q(0))  \qtq{and} \tb(\lambda;q(t)) = e^{-4i\lambda^4t} \tb(\lambda;q(0)).
\end{equation}

This analysis leads to the idea that $\ta(\lambda;q)$ encodes all the conserved quantities for the flow; in particular, \eqref{S structure} shows that it captures the modulus of $\tb(\lambda;q)$.  For many integrable systems, this has been shown to be the case.  However, if we consider the algebraic soliton $q_a$ defined in \eqref{algae}, we find that
\begin{equation}\label{Psi for alg}
\Psi(x;\lambda) = 
	e^{-i\lambda^2x\sigma_3} + \tfrac{4\lambda}{4\lambda^2+1} \begin{bmatrix} \frac{2i\lambda}{x-i} & \tfrac{e^{ix/2}}{x-i} \\
			-\tfrac{e^{-ix/2}}{x+i} & -\tfrac{2i\lambda}{x+i} \end{bmatrix}
	e^{-i\lambda^2x\sigma_3}
\end{equation}
is the solution to \eqref{Psi minus} for all $\lambda\in\C\setminus\{\pm \tfrac i2\}$.  This shows that $\ta(\lambda;q_a)\equiv 1$ and $\tb(\lambda;q_a)\equiv 0$ for all $\lambda^2\in\R\setminus\{-\frac14\}$.

We must ask about the significance of the exceptional values of $\lambda$.  There is a very compelling argument that they are meaningless.  As we will show more fully below, $\ta(\lambda;q)$ extends to a holomorphic function in the first quadrant and belongs to Nevanlinna class; as such, the function is determined by its a.e. boundary values.  Thus $\ta(\lambda;q_a)\equiv 1$ as a holomorphic function in the first quadrant.

One valuable approach to analyzing $\ta(\lambda;q)$, including showing that it is holomorphic, is to use that it is given by a certain Fredholm determinant:
\begin{align}\label{E:JP I}
\ta(\lambda ;q) = \det\bigl[ L_0^{-1} L \bigr]= \det\bigl[1 - \lambda^2  (-i\lambda^2 - \p)^{-1}q (-i \lambda^2 + \p)^{-1} \bar q \bigr],
\end{align}
valid for all $\lambda$ in the (open) first quadrant and $q\in\Schwartz$.  We will refer to this equality as a Jost--Pais identity, honoring \cite{MR0044404}, and will prove it in Section~\ref{S:3}.

If we fully accept our first answer regarding the exceptional values of $\lambda$, then there can be no hope of proving Theorem~\ref{t:DNLS equi} through $\ta(\lambda;q)$ alone: it cannot distinguish $q_a$ from zero, nor from any rescaling of $q_a$ via \eqref{e:scaling}.  As $\ta(\lambda;q)$ serves as a generating function for the polynomial conserved quantities, these functionals are also incapable of differentiating.

One of the morals that we have gleaned from the work of Bahouri--Perelman \cite{BP} is that $\arg[\ta(\lambda;q)]$, when properly interpreted, holds crucial information.  This is difficult to explain directly for $q_a$.  So imagine instead an approximating sequence of Schwartz functions; specifically, a sequence of traditional (bright) solitons, which are smooth with exponential decay.  For such solitons, the function $\ta(\lambda;q)$ is given by a single Blaschke factor whose zero is inside the open first quadrant.  Evidently, this zero forms a branch point for $\arg[\ta(\lambda;q)]$.  As the sequence approaches the algebraic soliton, this point approaches the point $i/2$ on the imaginary axis.  In this way, we find it instructive to view  the exceptional point as the vestige of a branch-point for $\arg[ \ta(\lambda;q)]$.

It is also true that the Lax operator $L(\lambda;q_a)$ has an eigenvalue at the point $\lambda=i/2$; indeed, the eigenfunction coincides with the residue of the function \eqref{Psi for alg} at this point.   The paper \cite{MR2202901} demonstrates that this eigenvalue exhibits a striking stability: under a broad class of perturbations (of either sign), the eigenvalue moves into the first quadrant; it does not dissolve into the continuous spectrum.  This further reinforces our interpretation of the exceptional point as the relic of zeros of $\ta(\lambda;q)$, or equivalently branch points of $\arg[\ta(\lambda;q)]$, that have moved to the boundary.

Once one believes that $\ta(\lambda;q)$ does encode enough information to prove equicontinuity, then one must also believe that an analogue of Theorem~\ref{t:DNLS equi} holds for the whole DNLS hierarchy and indeed, for any flow conserving $\ta(\lambda;q)$.  This is in fact our principal result, of which Theorem~\ref{t:DNLS equi} is an elementary corollary:

\begin{thrm}\label{t:Main}
Let $Q\subseteq \Schwartz$ be \(L^2\)-bounded and equicontinuous.  Then
\begin{align}\label{QsI}
Q_* = \bigl\{q\in \Schwartz: \ta(\lambda;q)\equiv \ta(\lambda;\tilde q) \text{ for some $\tilde q\in Q$}\bigr\}
\end{align}
is also \(L^2\)-bounded and equicontinuous.
\end{thrm} 

We write $\ta(\lambda;q)\equiv \ta(\lambda;\tilde q)$ to emphasize that equality holds as holomorphic functions on the open first quadrant. Clearly this is guaranteed as soon as the two functions agree on a sufficiently rich set, for example, the ray $\lambda\in\{\sqrt{\kappa}\;\! e^{i\pi/4} : \kappa>0\}$ employed in \cite{KNV}.  That paper also restricts attention to individual connected components of $Q_*$ (there called $Q_{**}$), their rationale being that well-posed flows remain within a single such connected component.

Note that Theorem~\ref{t:Main} deliberately only addresses ensembles of Schwartz-class functions. As we have observed earlier, this assertion would not be true if we required only $Q\subseteq L^2$ because the function $q\mapsto\ta(\lambda;q)$ is unable to distinguish rescalings of the algebraic soliton from one another.  Indeed, $\ta(\lambda;q)$ cannot even distinguish an algebraic soliton from zero!  On the other hand, as Lemma~\ref{l:COM} shows, $\ta(\lambda;q)$ does encode $M(q)$ for Schwartz functions. 

Here, we must acknowledge the deft formulation of the conjectures in \cite{KNV}: Requiring the initial data to be Schwartz ensures that the function $\ta(\lambda;q)$ retains enough information about $q$ to be useful, while employing the $L^2$ topology guarantees robust consequences for the well-posedness problem.  Moreover, once well-posedness is proved, the density of Schwartz space allows us to deduce important properties for all $H^s$ solutions.

Let us now turn to the question of how Theorem~\ref{t:Main} will be proved.  While the argument is ultimately quantitative in nature, for the sake of clarity, our synopsis here will be purely qualitative.  We will employ Littlewood--Paley projections to decompose $q$ into frequency pieces; these are defined in \eqref{LP defn}.

Suppose Theorem~\ref{t:Main} were to fail.  It is not difficult to show that $Q_*$ is bounded in $L^2$ and so this failure must be witnessed by a sequence of functions $q_n$ that is bounded but not equicontinuous.  Moreover, this sequence is accompanied by another sequence $\tilde q_n\in Q$ that is equicontinuous and satisfies $\ta(\lambda;q_n)\equiv \ta(\lambda;\tilde q_n)$, whence $M(q_n) = M(\tilde q_n)$.

Due to the failure of equicontinuity, for fixed $N$ and $n$ large, $P_{> N} q_n$ must carry non-trivial $L^2$ norm.  From this and a simple pigeonhole argument, we may find a wide ($n$-dependent) frequency window $\eps^3 N_1<|\xi|< \eps^{-3} N_1$ with three properties:
\begin{equation}\label{BS0}
\| P_{>\eps^3 N_1} \tilde q_n \| \ll 1, \quad \| P_{>\eps^{-3} N_1} q_n \| \gtrsim 1, \qtq{and} \| P_{\eps^3 N_1<\cdot\leq \eps^{-3} N_1} q_n \| \ll 1
\end{equation}
valid for any $n$ sufficiently large.  

We now focus on a particular choice of spectral parameter, namely, $\lambda_1 = \sqrt{iN_1}$.  As $\tilde q_n$ essentially vanishes at such high frequencies, we can understand $\ta(\lambda_1;\tilde q_n)$ very well; indeed,
\begin{align}\label{BS1}
\arg[ \ta(\lambda_1;\tilde q_n) ] \approx -\tfrac12 M(\tilde q_n).
\end{align}
Because $\lambda_1$ is well separated from the frequency regions inhabited by $q_n$ we may likewise understand $\arg[\ta(\lambda_1; q_n)]$ as the sum of the low- and high-frequency contributions:   
\begin{equation}\label{BS2}
\arg[ \ta(\lambda_1; q_n) ] \approx \arg[ \ta(\lambda_1; P_{\leq N_1} q_n) ] + \arg[ \ta(\lambda_1; P_{> N_1} q_n) ] .
\end{equation}
By analyzing these summands, we will deduce that
\begin{equation}\label{BS3}
\arg[ \ta(\lambda_1; q_n) ] \approx -\tfrac12 M(P_{\leq N_1} q_n).
\end{equation}
This yields a contradiction: as $\ta(\lambda;q_n)\equiv \ta(\lambda;\tilde q_n)$, it follows that that
$$
M(P_{> N_1} q_n) \approx  M(q_n) - M(P_{\leq N_1} q_n) \approx M(\tilde q_n) + 2 \arg[ \ta(\lambda_1, q_n) ] \approx 0,
$$
which is inconsistent with the middle condition in \eqref{BS0}.

We have been deliberately vague about the precise meaning of $\arg[ \ta(\lambda;q) ]$ in this outline; it is by no means trivial.  In fact, we shall only be employing this argument when $\ta(\lambda;\tilde q_n)$ has no zeros in the sector $\{\lambda\in \C: \tfrac\pi8<\arg \lambda<\frac\pi2\}$.  In this case, 
there is no trouble in choosing the correct branch of $\log[\ta(\lambda;q)]$ and the above provides a relatively faithful account of what is done in Section~\ref{S:5}.

Suppose now that $\ta(\lambda;\tilde q_n)$ has a single zero in this sector.  Then so too does $\ta(\lambda;q_n)$.  We will employ a B\"acklund transformation to remove this zero from both.  We give a brief review of this transformation at the beginning of Section~\ref{S:6} following \cite{MR3702542}, which we also recommend for further discussion and historical references.

Applying the B\"acklund transformation to $q_n$ and $\tilde q_n$ yields two new functions $\bB(q_n)$ and $\bB(\tilde q_n)$; moreover,
$$
\ta\bigl(\lambda;\bB(q_n)\bigr)\equiv \ta\bigl(\lambda;\bB(\tilde q_n)\bigr)
$$
and this function is zero-free in the sector $\{\lambda\in \C: \tfrac\pi8<\arg \lambda<\frac\pi2\}$.  In Section~\ref{S:6}, we prove that  the sequence $\bB(\tilde q_n)$ inherits equicontinuity from $\tilde q_n$.  This allows us to infer equicontinuity of $\bB(q_n)$ from the argument presented above.  This does not suffice: we must show that the sequence $q_n$ is equicontinuous.

In general, we cannot expect to infer equicontinuity of $q_n$ from $\bB(q_n)$; the location of the zero that is removed matters very much here. However, by Proposition~\ref{p:Zeros}, the location of this zero is very strongly constrained by the equicontinuity of $\tilde q_n$.  This will allow us to demonstrate that equicontinuity of both $\tilde q_n$ and $\bB(q_n)$ ensures equicontinuity of $q_n$.  This settles the case of one zero in the sector.

While one could imagine applying iterated B\"acklund transformations to reduce profiles $q_n$ with more zeros to the argument above, we find it more convenient to argue by induction on the number of zeros.  The details are provided in Section~\ref{S:7}.

Finally, let us provide a more thorough description of the role of each section in completing these arguments.

In Section~\ref{S:2}, we study the function $\ta(\lambda;q)$ through the determinantal representation \eqref{E:JP I}.  We will use this approach to analyze how this quantity behaves when $-i\lambda^2\in [0,\infty)$ is far from the Fourier support of $q$; see Proposition~\ref{p:Winding}.  This is used to justify \eqref{BS1} and \eqref{BS3}.

The expression \eqref{BS2} represents a form of asymptotic factorization.  This is proved in Proposition~\ref{p:Factor}.  The fact that this represents decoupling under separation of scales (rather than via translation) distinguishes it from the related factorization results in \cite{BP,killip2020orbital}.

In Section~\ref{S:3}, we analyze $\ta(\lambda;q)$ from the point of view of (time-independent) scattering theory.  We prove that $\ta(\lambda;q)$ depends on both $\lambda$ and $q$ in a smooth fashion (in the \emph{closed} first quadrant) and use this to verify the Jost--Pais identity.  We also prove Proposition~\ref{p:Zeros}, which documents how equicontinuity constrains the location of any zeros of $\ta(\lambda;q)$.

In Section~\ref{S:4}, we prove the trace formula \eqref{Trace} for general Schwartz $q$. The parallel result for those Schwartz $q$ without spectral singularities was an important tool in \cite{BP}. One application of the trace formula is that it provides a (finite) upper bound on the total number of zeros that may lie within the sector $\{\lambda\in \C: \tfrac\pi8<\arg \lambda<\frac\pi2\}$.  We also use it in Lemma~\ref{4.3} to derive a lower bound on $\ta$. 

Section~\ref{S:5} proves Theorem~\ref{t:Main} for sets $Q$ where $\ta(\lambda;q)$ has no zeros in the sector $\{\lambda\in \C: \tfrac\pi8<\arg \lambda<\frac\pi2\}$.  This constitutes the base step of our inductive argument.

We begin Section~\ref{S:6} by reviewing the B\"acklund transform associated to \eqref{DNLS}, closely following \cite{MR3702542}.  The principal novelty of this section is Proposition~\ref{p:Und} which shows (to put it loosely) that both the B\"acklund transform and its inverse preserve $L^2$-equicontinuity. The paper ends with Section~\ref{S:7}, which completes the inductive step of our argument.

\subsection{Notation}  As it will take us some time to develop the necessary prerequisites for proving the Jost--Pais identity \eqref{E:JP I}, we must introduce an alternate notation for the determinant appearing there.  Considering that the spectral parameter $\lambda$ only appears squared here, it is natural to adopt a different parameterization based on $k=\lambda^2\in \C^+=\{k\in\C : \Im k >0\}$, with the happy consequence that we shall be able to employ the well-documented theory of holomorphic functions in the half plane.

Given $k\in\C^+$ and $q\in L^2(\R)$, we first define operators
\[
\Lambda(k;q) = \bigl(\tfrac1ik - \p\bigr)^{-\frac12}q\bigl(\tfrac1ik + \p\bigr)^{-\frac12}\qtq{and}\Gamma(k;q) = \bigl(\tfrac1ik + \p\bigr)^{-\frac12}\bar q\bigl(\tfrac1ik - \p\bigr)^{-\frac12}.
\]
From \cite[Lemma 4.1]{MR3820439} we have
\begin{equation}\label{I2}
\|\Lambda\|_{\I_2}^2 = \|\Gamma\|_{\I_2}^2 \sim \int_\R\log\bigl(4 + \tfrac{|\xi|^2}{|\Im k|^2}\bigr)\frac{|\hat q(\xi - 2\Re k)|^2}{\sqrt{4|\Im k|^2 + |\xi|^2}}\,d\xi\lesssim \tfrac1{\Im k}\|q\|_{L^2}^2,
\end{equation}
where $\I_2$ denotes the class of Hilbert--Schmidt operators.  This estimate shows that for any \(q\in L^2\) the perturbation determinant
\begin{align}\label{E:a defn}
a(k;q) := \det\bigl[1 - k\Lambda(k;q)\Gamma(k;q)\bigr]
\end{align}
is well-defined and analytic for \(k\in \C^+\).  This is not precisely the determinant appearing in \eqref{E:JP I} but it is easy to prove that they agree: As
\begin{equation*}
\bigl\| k  (-ik - \p)^{-1}q (-ik + \p)^{-1}\bar q \bigr\|_{\I_1} \leq |k| \cdot \| q \|_{L^2}^2 \int_\R |k-\xi|^{-2}\,d\xi <\infty,
\end{equation*}
we may permute the factors and so deduce that
\begin{align}\label{E:alt a defn}
a(k;q) = \det\bigl[1 - k  (-ik - \p)^{-1}q (-ik + \p)^{-1}\bar q \bigr].
\end{align}

The virtue of adopting \eqref{E:a defn} as our primary definition of the perturbation determinant is that it is better suited to estimating the contribution of each Littlewood--Paley piece of $q$ and $\bar q$.

Our notation for the Littlewood--Paley decomposition is standard: For each $N\in 2^\Z$ we define $P_N$ as a smooth localization (based on a partition of unity) to those frequencies $\xi\in\R$ with $\frac N2\leq |\xi|\leq 2N$. We then define
\begin{equation}\label{LP defn}
P_{>N} = \sum_{K>N} P_K \qtq{and} P_{\leq N} = I - P_{>N}.
\end{equation}
Note that if $q$ is Schwartz, then so are $q_N:=P_N q$, $q_{\leq N}:=P_{\leq N}q$, and $q_{>N}:=P_{>N }q$.

In light of \eqref{M_of_q}, we will avoid using $M$ as a dyadic frequency parameter.  Rather, given \(M>0\), we define the ball
\begin{align}\label{BM}
B_M = \bigl\{q\in L^2:\|q\|_{L^2}^2\leq M\bigr\}.
\end{align}

This vocabulary allows us to give a more quantitative Fourier-based characterization of equicontinuity:  \(Q\subseteq B_M\) is equicontinuous if and only if, for any \(\epsilon>0\), we may find \(N = N(\epsilon,Q)\in 2^\Z\) so that
\begin{equation}\label{HF-UNI}
\sup_{q\in Q}\|q_{>N}\|_{L^2}^2\leq \epsilon^2 M.
\end{equation}
This equivalence is quite elementary to verify; see \cite[\S5]{MR3990604} for details.

As our last piece of notation, we recall the Pauli matrices:
\begin{align}\label{Pauli}
\sigma_1 = \begin{bmatrix}0&1\\1& 0\end{bmatrix},\quad
\sigma_2 = \begin{bmatrix}0&-i\\i&0\end{bmatrix},\qtq{and}
\sigma_3 = \begin{bmatrix}1&0\\0&-1\end{bmatrix},
\end{align}
the last of which was seen already in \eqref{L0 defn}.

\subsection*{Acknowledgements} R. K. was supported by NSF grant DMS-1856755 and M.~V. by grants DMS-1763074 and DMS-2054194.

\section{The perturbation determinant}\label{S:2}

In this section we analyze the perturbation determinant $a(k;q)$ defined in \eqref{E:a defn}.  We begin with some basic estimates for the determinant:

\begin{lem}\label{l:Det}
Let \(A\in \I_1\). Then
\begin{align}
\bigl|\det(1 + A)\bigr| &\leq \exp(\|A\|_{\I_1}),\label{det BASIC}\\
\bigl|\det(1 + A) - 1\bigr|&\leq \|A\|_{\I_1}\exp(\|A\|_{\I_1}),\label{det diff -1}\\
\bigl|\det(1 + A) - \exp\{\tr A\}\bigr|&\leq \tfrac12\|A\|_{\I_2}^2\exp(\|A\|_{\I_1}).\label{det diff 0}
\end{align}

Further, if \(B\in \I_1\) and we have
\[
1 + \|A\|_{\I_1} + \|B\|_{\I_1}\leq M,
\]
then
\begin{align}
\det(1 + A)\det(1 + B) &= \det\bigl[(1 + A)(1 + B)\bigr],\label{det prod}\\
\bigl|\det(1 + A) - \det(1 + B)\bigr|&\leq e^M\|A - B\|_{\I_1},\label{det diff 1}\\
\bigl|\det(1 + A) - \det(1 + B)\bigr|&\leq e^{2M^2}\Bigl(\bigl|\tr(A - B)\bigr| + \|A - B\|_{\I_2}\Bigr).\label{det diff 2}
\end{align}
\end{lem}
\begin{proof}
The estimate \eqref{det diff 0} is proved in \cite[Lemma~3.2]{KNV} by combining the Weyl inequalities (cf. \cite[Theorem~1.15]{MR2154153}) with the identity
\[
\det(1+A) = 1 + \sum_{n=1}^\infty \frac1{n!}\sum_{\substack{i_1,\dots,i_n\\\text{distinct}}}\lambda_{i_1}\lambda_{i_2}\dots\lambda_{i_n}.
\]
Here, \(\lambda_j\) are the non-zero eigenvalues of \(A\) repeated according to algebraic multiplicity. This expression also yields \eqref{det BASIC} and \eqref{det diff -1}:
\begin{align*}
\LHS{det BASIC} &\leq \sum_{n=0}^\infty \frac1{n!}\Bigl(\sum_j|\lambda_j|\Bigr)^n\leq \RHS{det BASIC},\\
\LHS{det diff -1} &\leq \sum_{n=1}^\infty \frac1{n!}\Bigl(\sum_j|\lambda_j|\Bigr)^n\leq \|A\|_{\I_1}\sum_{\ell=0}^\infty \frac{\|A\|_{\I_1}^\ell}{(\ell+1)!}\leq \RHS{det diff -1}.
\end{align*}

The identity \eqref{det prod} is proved in~\cite[Theorem~3.5]{MR2154153} and the estimate \eqref{det diff 1} in~\cite[Theorem~3.4]{MR2154153}.

For \eqref{det diff 2} we employ the regularized determinant
\[
\det\!_2(1 + A) = \det(1 + A)\exp\{-\tr A \},
\]
and from~\cite[Theorem~9.1]{MR2154153} we have
\[
\bigl|\det\!_2(1 + A) - \det\!_2(1 + B)\bigr|\leq e^{\frac12 M^2}\|A - B\|_{\I_2}.
\]
We then combine this with \eqref{det BASIC} to obtain
\begin{align*}
\LHS{det diff 2}&\leq \bigl|\det(1 + A) - \exp\bigl\{\tr(A-B)\bigr\}\det(1 + B)\bigr|\\
&\quad + \bigl|\exp\bigl\{\tr(A - B)\bigr\} -1\bigr| \bigl|\det(1 + B)\bigr|\\
&\leq e^M \bigl|\det\!_2(1 + A) - \det\!_2(1 + B)\bigr| + e^{2M}\bigl|\tr(A - B)\bigr|\\
&\leq e^{2M^2}\Bigl(\bigl|\tr(A - B)\bigr| + \|A - B\|_{\I_2}\Bigr). \qedhere
\end{align*}
\end{proof}

Next, we record some useful estimates for frequency-localized potentials. For \(N\in 2^\Z\), we introduce the notation
\[
\Lambda_N = \Lambda(k;q_N),\qquad\Gamma_N = \Gamma(k;q_N),
\]
with similar definitions for \(\Lambda_{\leq N}\), \(\Gamma_{\leq N}\), etc. Operator estimates for $\Gamma$ can be deduced from those for \(\Lambda\) (and vice versa) by conjugating with a spatial reflection. Thus, we will typically only state these bounds in terms of either \(\Lambda\) or \(\Gamma\) in what follows.

The following estimates are essentially identical to those in~\cite[Lemma~2.4]{KNV}:

\begin{lem}\label{l:Loc} For \(k\in \C^+\) we have the estimate
\begin{equation}
\sqrt{\Im k}\|\Lambda\|_{\op} \leq \sqrt{\Im k}\|\Lambda\|_{\I_2}\lesssim \|q\|_{L^2}.\label{OP-BASIC}
\end{equation}
Further, if \(k = i\kappa\in i\R^+\) and \(N\in 2^\Z\), we have the estimates
\begin{align}
\sqrt{\kappa}\|\Lambda_{\leq N}\|_{\op} &\lesssim \sqrt{\tfrac N\kappa}\|q_{\leq N}\|_{L^2},\label{Op-Low}\\
\sqrt{\kappa}\|\Lambda_{>N}\|_{\I_2}  &\lesssim \sqrt{\tfrac{\kappa}N\log\bigl(4 + \tfrac{N^2}{\kappa^2}\bigr)}\|q_{>N}\|_{L^2}.\label{I2-Hi}
\end{align}
\end{lem}
\begin{proof}
The estimate \eqref{OP-BASIC} follows from \eqref{I2}. Using Bernstein's inequality, for \(k = i\kappa\) we may bound
\[
\|\Lambda_N\|_{\op}\lesssim \tfrac1\kappa\|q_N\|_{L^\infty}\lesssim \tfrac{\sqrt N}{\kappa}\|q_N\|_{L^2}.
\]
We then sum to obtain \eqref{Op-Low}. The estimate \eqref{I2-Hi} follows from \eqref{I2} and the fact that
\[
\|\Lambda_{>N}\|_{\I_2}^2 \sim \sum\limits_{K>N}\|\Lambda_K\|_{\I_2}^2.\qedhere
\]
\end{proof}

Our first application of Lemmas~\ref{l:Det} and \ref{l:Loc} is a description of the behavior of $a(i\kappa;q)$ for $\kappa$ large and small.  Such asymptotics were also analyzed in \cite{BP} via similar operator-theoretic means, although their primary focus was on the case $q\in H^{1/2}(\R)$.

\begin{prop}\label{p:Winding}Let \(0<\epsilon<\frac12\) and \(q\in B_M\), where $B_M$ is as defined in \eqref{BM}.

\smallskip

\noindent \textup{(i)} If \(N\in 2^\Z\) is chosen so that
\begin{equation}\label{HF}
\|q_{>N}\|_{L^2}^2\leq \epsilon^2 M,
\end{equation}
then, for any \(\kappa\geq\frac N{\epsilon^2}\), we have
\begin{equation}\label{No high winding}
\Bigl|a(i\kappa;q) - e^{-\tfrac i2\|q\|_{L^2}^2}\Bigr|\lesssim_M \epsilon^2.
\end{equation}

\smallskip

\noindent \textup{(ii)} If \(N\in 2^\Z\) is chosen so that
\begin{equation}\label{LF}
\|q_{\leq N}\|_{L^2}^2\leq \epsilon^2 M,
\end{equation}
then, for any \(0<\kappa\leq\epsilon^2N\), we have
\begin{equation}\label{No low winding}
\Bigl|a(i\kappa;q) - 1\Bigr|\lesssim_M \epsilon^2\log\bigl(\tfrac1\epsilon\bigr).
\end{equation}
\end{prop}

\begin{proof}(i) Taking \(k = i\kappa\) for \(\kappa\geq \frac N{\epsilon^2}\), we apply \eqref{OP-BASIC} and \eqref{Op-Low} to bound
\begin{align*}
\sqrt{\kappa}\|\Lambda\|_{\op} \leq \sqrt{\kappa}\|\Lambda_{>N}\|_{\I_2} + \sqrt{\kappa}\|\Lambda_{\leq N}\|_{\op} \lesssim \|q_{>N}\|_{L^2}+\sqrt{\tfrac N{\kappa}}\|q\|_{L^2}\lesssim \epsilon\sqrt{M}.
\end{align*}
We then apply \eqref{OP-BASIC} and \eqref{det diff 0} to obtain
\begin{align*}
\Bigl|a(i\kappa;q) - \exp\bigl\{-\tr\bigl(i\kappa\Lambda\Gamma\bigr)\bigr\}\Bigr|&\lesssim \kappa^2\|\Lambda\|_{\op}^2\|\Gamma\|_{\I_2}^2\exp\Bigl(\kappa\|\Lambda\|_{\I_2}\|\Gamma\|_{\I_2}\Bigr)\\&\lesssim \epsilon^2M^2 \exp\bigl(CM\bigr),
\end{align*}
for some constant \(C>0\).

A computation (see \cite[Lemma~2.2]{KNV}) shows that 
\begin{align}\label{trace}
\tr\bigl\{k\Lambda(f)\Gamma(h)\bigr\}= \int_\R\frac{ik}{2k+\xi}\hat f(\xi) \overline{\hat h}(\xi)\, d\xi.
\end{align}
Using \eqref{trace} with $f=h=q$ and $k=i\kappa$, we have
\[
\tfrac i2\|q\|_{L^2}^2 - \tr\bigl\{i\kappa\Lambda\Gamma\bigr\} = \int_\R \frac{i\xi|\hat q(\xi)|^2}{2(2i\kappa + \xi)}\,d\xi,
\]
and hence for $\kappa\geq \frac{N}{\eps^2}$,
\[
\Bigl|\tfrac i2\|q\|_{L^2}^2 - \tr\bigl\{i\kappa\Lambda\Gamma\bigr\}\Bigr|\lesssim \tfrac{N}{\kappa}\|q_{\leq N}\|_{L^2}^2 + \|q_{>N}\|_{L^2}^2\lesssim \epsilon^2 M.
\]

Combining these bounds, we obtain \eqref{No high winding}.

\smallskip

\noindent (ii) Assuming $k=i\kappa$ with $0<\kappa\leq \eps^2N$, and combining \eqref{OP-BASIC} with \eqref{I2-Hi}, we obtain
\begin{align*}
\sqrt{\kappa}\|\Lambda\|_{\I_2} &\leq \sqrt{\kappa}\|\Lambda_{\leq N}\|_{\I_2} + \sqrt{\kappa}\|\Lambda_{>N}\|_{\I_2}\\
&\lesssim \|q_{\leq N}\|_{L^2} + \sqrt{\tfrac{\kappa}N\log\bigl(4 + \tfrac{N^2}{\kappa^2}\bigr)}\|q\|_{L^2}\\
&\lesssim \epsilon \sqrt{\log\bigl(\tfrac1\epsilon\bigr)}\sqrt M.
\end{align*}
We then apply \eqref{det diff -1} to bound
\[
\Bigl|a(i\kappa;q) - 1\Bigr|\lesssim \kappa\|\Lambda\|_{\I_2}\|\Gamma\|_{\I_2}\exp\Bigl(\kappa\|\Lambda\|_{\I_2}\|\Gamma\|_{\I_2}\Bigr)\lesssim \epsilon^2\log\bigl(\tfrac1{\epsilon}\bigr) M\exp\bigl(CM\bigr),
\]
for some \(C>0\), which gives us \eqref{No low winding}.
\end{proof}

Proposition~\ref{p:Winding} provides an alternate path to proving the following result shown already in \cite{BP}:

\begin{cor}\label{c:Limits}
For all \(q\in L^2\) we have
\[
\lim_{\kappa\to0}a(i\kappa;q) = 1\qtq{and}\lim_{\kappa\to\infty}a(i\kappa;q) = e^{-\frac i2\|q\|_{L^2}^2}.
\]
\end{cor}

From Proposition~\ref{p:Winding}(i), we also obtain the following:

\begin{cor}\label{c:NHWU}
Let \(\epsilon>0\) and \(Q\subseteq B_M\) be equicontinuous. Given \(N\in 2^\Z\) so that \(Q\) satisfies \eqref{HF-UNI} and \(\kappa\geq\frac N{\epsilon^2}\), we have
\begin{equation}\label{No high winding-UNI}
\sup_{q\in Q} \,\Bigl|a(i\kappa;q) - e^{-\tfrac i2\|q\|_{L^2}^2}\Bigr|\lesssim_M \epsilon^2.
\end{equation}
\end{cor}

In our proof of Theorem~\ref{t:Main}, we will employ the estimates of Proposition~\ref{p:Winding} in conjunction with the following factorization property:

\begin{prop}[Factorization]\label{p:Factor}
Let \(0<\epsilon\leq\frac14\) be a dyadic number and \(q\in B_M\).  If \(N\in 2^\Z\) is such that
\begin{equation}\label{Gap}
\| q_{\epsilon^3 N<\cdot\leq \frac N{\epsilon^3}}\|_{L^2}\leq \epsilon\|q\|_{L^2},
\end{equation}
then
\begin{equation}\label{Factor}
\Bigl|a(i\kappa;q) - a(i\kappa;q_{\leq N})\,a(i\kappa;q_{>N})\Bigr|\lesssim_M\epsilon^2 \quad\text{uniformly for $\kappa>0$}.
\end{equation}
\end{prop}

\begin{proof} Let \(k = i\kappa\) for \(\kappa>0\).

If \(0<\kappa\leq N\), we may apply \eqref{OP-BASIC} and \eqref{I2-Hi} with the hypothesis \eqref{Gap} to estimate
\begin{align*}
\sqrt{\kappa}\|\Lambda_{>N}\|_{\I_2} &\leq \sqrt{\kappa}\|\Lambda_{N<\cdot\leq \frac N{\epsilon^3}}\|_{\I_2} + \sqrt{\kappa}\|\Lambda_{> \frac N{\epsilon^3}}\|_{\I_2}\\
&\lesssim \|q_{\epsilon^3N<\cdot\leq \frac N{\epsilon^3}}\|_{L^2} + \sqrt{\tfrac{\kappa\epsilon^3}N\log\bigl(4 + \tfrac{N^2}{\epsilon^6\kappa^2}\bigr)}\|q\|_{L^2}\\
&\lesssim \epsilon\sqrt M.
\end{align*}
Conversely, if \(\kappa\geq N\) then from \eqref{OP-BASIC}, \eqref{Op-Low}, and \eqref{Gap} we have
\begin{align*}
\sqrt{\kappa}\|\Gamma_{\leq N}\|_{\op} &\leq \sqrt{\kappa}\|\Gamma_{\epsilon^3 N<\cdot\leq N}\|_{\I_2}  + \sqrt{\kappa}\|\Gamma_{\leq \epsilon^3N}\|_{\op}\\
&\lesssim \|q_{\epsilon^3N<\cdot\leq \frac N{\epsilon^3}}\|_{L^2} + \sqrt{\tfrac{\epsilon^3N}{\kappa}}\|q\|_{L^2}\\
&\lesssim \epsilon\sqrt M.
\end{align*}
Combining these bounds, we deduce that
\begin{equation}\label{Gapped}
\kappa\|\Lambda_{>N}\|_{\I_2}\|\Gamma_{\leq N}\|_{\op} +\kappa\|\Gamma_{>N}\|_{\I_2}\|\Lambda_{\leq N}\|_{\op} \lesssim  \epsilon^2 M \quad\text{uniformly for \(\kappa>0\).}
\end{equation}

From \eqref{trace}, we get
\[
\tr\bigl\{i\kappa \Lambda_{\leq N}\Gamma_{>N}\bigr\} = \int_\R \tfrac{i\kappa}{2\kappa - i\xi} \overline{\hat q_{>N}(\xi)} \hat q_{\leq N}(\xi)\,d\xi.
\]
Frequency support considerations then allow us to use \eqref{Gap} to bound
\begin{equation}\label{Gappy}
\bigl|\tr\bigl\{i\kappa \Lambda_{\leq N}\Gamma_{>N}\bigr\}\bigr|\lesssim \|q_{\epsilon^3N<\cdot\leq N}\|_{L^2}\|q_{N<\cdot\leq \frac N{\epsilon^3}}\|_{L^2}\lesssim \epsilon^2 M.
\end{equation}

We now apply \eqref{det prod}, \eqref{det diff 1}, and \eqref{Gapped} to bound
\begin{align*}
&\Bigl|a(i\kappa;q_{\leq N})\,a(i\kappa;q_{>N}) - \det\bigl[1 - i\kappa\Lambda_{\leq N}\Gamma_{\leq N} - i\kappa\Lambda_{>N}\Gamma_{>N}\bigr]\Bigr|\\
&\qquad\lesssim \kappa^2\|\Lambda_{\leq N}\|_{\op}\|\Gamma_{\leq N}\|_{\op}\|\Lambda_{>N}\|_{\I_2}\|\Gamma_{>N}\|_{\I_2}\exp\bigl(CM^2\bigr)\\
&\qquad\lesssim \epsilon^4M^2\exp\bigl(CM^2\bigr),
\end{align*}
for some \(C>0\). Second, we use \eqref{det diff 2}, \eqref{Gapped}, and \eqref{Gappy} to bound
\begin{align*}
&\Bigl|a(i\kappa;q) - \det\bigl[1 - i\kappa\Lambda_{\leq N}\Gamma_{\leq N} - i\kappa\Lambda_{>N}\Gamma_{>N}\bigr]\Bigr|\\
&\qquad \lesssim \Biggl\{\bigl|\tr\bigl\{i\kappa\Lambda_{\leq N}\Gamma_{>N}\}\bigr| + \bigl|\tr\bigl\{i\kappa\Lambda_{> N}\Gamma_{\leq N}\}\bigr| + \kappa\|\Lambda_{\leq N}\|_{\op}\|\Gamma_{>N}\|_{\I_2}\\
&\qquad\qquad + \kappa\|\Lambda_{>N}\|_{\I_2}\|\Gamma_{\leq N}\|_{\op}\Biggr\}\exp\bigl(CM^2\bigr)\\
&\qquad \lesssim \epsilon^2 M\exp\bigl(CM^2\bigr),
\end{align*}
which completes the proof of \eqref{Factor}.
\end{proof}

\section{The reciprocal of the transmission coefficient}\label{S:3}

In this section, we analyze $\ta(\lambda;q)$ which is defined through the asymptotic behavior (as $x\to\pm\infty$) of solutions to the Kaup--Newell system
\begin{equation}\label{KN} 
\psi' = -i\sigma_3\lambda^2\psi + \lambda\begin{bmatrix}0&q\\-\bar q&0\end{bmatrix}\psi.
\end{equation}

We begin by discussing the key identities for $q\in \Test(\R)$.  In this case, it is trivial to see that there are matrix solutions $\Psi^\pm$ to \eqref{KN} satisfying
\begin{equation}\label{JOSTBC}
\Psi^\pm(x;\lambda) e^{i\lambda^2x\sigma_3}  = \Id \quad \text{for $\pm x$ sufficiently large}.
\end{equation}
Indeed, equality holds as soon as $x$ is large enough to lie beyond the support of $q$.  Moreover, for each \(x\in \R\), the map \(\lambda\mapsto \Psi^\pm(x;\lambda)\) is entire and \(\det\Psi^\pm = 1\).

Note that the solution $\Psi$ appearing in the introduction is precisely $\Psi^-$.  Moreover, when $q\in\Test(\R)$, it is evident that limit \eqref{S defn} exits and defines a matrix of determinant one.  To verify \eqref{S structure}, we need to employ certain symmetries of \eqref{KN}.

If \(\psi\) is a solution of \eqref{KN} then \(\sigma_3\psi\) is a solution with \(\lambda\) replaced by \(-\lambda\), and \(\sigma_1\sigma_3\bar\psi\) is also a solution with \(\lambda\) replaced by \(\bar \lambda\).  (Here we use the notations \eqref{Pauli}.)  By comparing asymptotic behavior, we deduce that
\begin{subequations}\label{SIMS}
\begin{align}
\Psi^\pm(x;\lambda) &= \sigma_3\Psi^\pm(x;-\lambda) \sigma_3,\\
\Psi^\pm(x;\lambda) &= \sigma_1\sigma_3\overline{\Psi^\pm(x;\bar\lambda)} \sigma_3\sigma_1.
\end{align}
\end{subequations}
The assertions \eqref{S structure} follow from this and the fact that $\det S(\lambda;q) =1$.

Moving forward, we would like to focus on the particular matrix element $\Psi^-_{11}$ of the Jost solution $\Psi^-$, which we will characterize as the solution of certain integral equations.  Our ultimate aim is to provide useful means of analyzing the key quantity $\ta(\lambda;q)$.  Initially, this will just be for $q\in\Test(\R)$.

The role of $\ta$ and $\tb$ as scattering coefficients is apparent from the relation
\begin{equation}\label{scat coeff}
 \Psi^-_{11}(x) = \ta e^{-i\lambda^2 x} + \tb e^{i\lambda^2 x} 
\end{equation}
for $x$ to the right of the support of $q$. Concretely, $1/\ta$ represents the amplitude of the transmitted wave, while $\tb/\ta$ represents the amplitude of the reflected wave. 

Our first approach is via a Fredholm integral equation.  From \eqref{scat coeff} and \eqref{JOSTBC}, we know that
\begin{align*}
\Psi^-_{21}, \qquad \bar q\,\Psi^-_{11}, \qtq{and}  \Psi^-_{11}-\ta(\lambda)e^{-i\lambda^2 x}
\end{align*}
are all square-integrable when $\lambda^2\in\C^+$; moreover, they satisfy
\begin{align*}
(-i\lambda^2+\partial) \Psi^-_{21} = - \lambda \bar q \Psi^-_{11} \qtq{and}  (-i\lambda^2-\partial)[\Psi^-_{11}-\ta(\lambda)e^{-i\lambda^2 x}] = - \lambda q\Psi^-_{21}.
\end{align*}
Thus, for $\lambda^2\in\C^+$ we have
\begin{align}\label{2 JP}
\Psi^-_{11}(x)-\ta(\lambda)e^{-i\lambda^2 x} = \lambda^2 (-i\lambda^2-\partial)^{-1} q (-i\lambda^2+\partial)^{-1} \bar q \,\Psi^-_{11}.
\end{align}
This representation will be used to prove the Jost--Pais identity \eqref{E:JP I} in Lemma~\ref{l:JP}.

A second approach is to represent $\Psi^-_{11}$ as the solution of the Volterra equation
\begin{equation}\label{p11}
\Psi_{11}^-(x) = e^{-i\lambda^2 x} - \lambda^2\int_{-\infty}^x\int_{-\infty}^s e^{-i\lambda^2(x-2s+y)}q(s) \bar q(y)\Psi_{11}^-(y)\,dy\,ds,
\end{equation}
which follows simply from the initial conditions: $\Psi^-_{11}(x)=e^{-i\lambda^2 x}$ and $\Psi^-_{12}(x)=0$ for every $x$ to the left of the support of $q\in\Test(\R)$.

While \eqref{p11} is perfectly satisfactory for constructing $\ta(\lambda;q)$ via
\begin{equation}\label{p11a}
\ta(\lambda;q) = \lim_{x\to\infty} e^{i\lambda^2 x} \Psi_{11}^-(x) = 1 - \lambda^2\int_{-\infty}^\infty \int_{-\infty}^s e^{i\lambda^2(2s-y)}q(s) \bar q(y)\Psi_{11}^-(y)\,dy\,ds
\end{equation}
and for verifying most of the properties we need, it is not well suited to describing the large-$\lambda$ asymptotics.  In order to give a similar representation that will serve all our needs, we adopt a change of unknowns introduced already in \cite{MR464963}: Setting
\begin{align*}
\gamma(x) = e^{\frac12im(x) \sigma_3}\begin{bmatrix}1&0\\-\frac12\bar q(x)&i\lambda\end{bmatrix}
	\begin{bmatrix}\Psi_{11}^-(x)\\\Psi_{21}^-(x)\end{bmatrix}
	\qtq{with}m(x) = \int^x_{-\infty} |q(y)|^2\,dy,
\end{align*}
we find that
\begin{align}\label{E:gamma'}
\gamma' = -i\sigma_3\lambda^2\gamma - i \begin{bmatrix}0&q e^{im}\\-re^{-im}&0\end{bmatrix}\gamma \qtq{where}r = \tfrac i2\bar q' + \tfrac14|q|^2\bar q.
\end{align}

\begin{rem}\label{R:2} In deriving the equation for $\gamma$, we used the following general result:
\begin{equation*}
P(\p_x + i\lambda^2\sigma_3 - Q) = (\p_x + i\lambda^2\sigma_3 - \widetilde Q)P
	\ \iff \ \widetilde QP = P' + i\lambda^2[\sigma_3,P] + PQ.
\end{equation*}
\end{rem}

Using the first component of the vector $\gamma$, we build $\phi(x) := e^{i\lambda^2 x}\gamma_1(x)$.  For this function, we have the following analogues of \eqref{p11} and~\eqref{p11a}:
\begin{gather}\label{g11}
\phi(x) = 1 + \int_{-\infty}^x\int_{-\infty}^s e^{2i\lambda^2(s-y)+im(s)-im(y)}q(s) r(y)\phi(y)\,dy\,ds, \\
	\label{g11a}
e^{\frac i2\|q\|_{L^2}^2} \ta(\lambda;q) = \lim_{x\to\infty} \phi(x).
\end{gather}
This is the representation that will allow us to prove the estimates we need.  Notice that in both \eqref{p11} and \eqref{g11}, the spectral parameter $\lambda$ only appears squared.  Thus, it is natural to adopt $k=\lambda^2$ with $\Im k\geq 0$ as our primary parameter (as in Section~\ref{S:2}); by $\sqrt{k}$, we will always mean the value in the first quadrant.

We will need two further notations:  We write $X$ for the Banach space of holomorphic functions \(f\colon\C^+\to \C\) such that $f$ and $f'$ extend continuously to $\p\C^+$ and for which
\begin{align}\label{X defn}
\|f\|_X = \sup_{k\in \overline{\C^+}}\Bigl(\<k\>|f(k)| + \<k\>^2|f'(k)|\Bigr)<\infty.
\end{align}
The space \(X\) is readily seen to be a closed subspace of the Hardy space \(H^\infty(\C^+)\). 
We also recall the notation $H^{4,4}(\R)$ denoting the completion of $\Schwartz$ in the norm
$$
\|q\|_{H^{4,4}}^2 = \|q^{(4)}\|_{L^2}^2 + \|\<x\>q'''\|_{L^2}^2 + \|\<x\>^{2}q''\|_{L^2}^2 + \|\<x\>^{3}q'\|_{L^2}^2 + \|\<x\>^{4}q\|_{L^2}^2.
$$

\begin{lem}\label{l:Jost-AUX}
The function
$
q\mapsto e^{\frac i2\|q\|_{L^2}^2} \ta(\sqrt k;q) - 1
$
extends from $q\in C^\infty_c(\R)$ to a real-analytic $X$-valued function of $q\in H^{4,4}(\R)$.
\end{lem}

\begin{proof}
First we construct the solution $\phi(x;k)$ by interpreting \eqref{g11} as the Volterra integral equation
\begin{align}\label{g11'}
\phi = 1+ A \phi \qtq{with} [A\phi](x) = \int_{-\infty}^x \! A(x,y;k) \phi(y)\,dy
\end{align}
and integral kernel
\begin{align*}
A(x,y;k) = \int_y^x e^{2ik(s-y)+im(s)-im(y)}q(s) r(y)\,ds.
\end{align*}

We will solve \eqref{g11'} by iteration, working in the Banach space $\mathcal C$ comprised of those bounded continuous functions $\varphi:\R\to\C$ that have limits as $x\to\pm\infty$, which we endow with the supremum norm  (or equivalently, in the Banach space of continuous functions on the two-point compactification $[-\infty,\infty]$).  
This small wrinkle is helpful in light of \eqref{g11a}.

Estimating the kernel either directly or after integrating by parts in $s$, we find
\begin{align*}
|A(x,y;k)| \lesssim \|q\|_{L^1} |r(y)| \qtq{and} |A(x,y;k)| \lesssim |k|^{-1} \Bigl[ \|q\|_{L^\infty} + \| (e^{im/2}q)' \|_{L^1}\Bigr] |r(y)|
\end{align*}
uniformly for $\Im k \geq 0$.  This shows $\|A\varphi\|_{L^\infty}\lesssim \langle k\rangle^{-1} \|\varphi\|_{L^\infty}$.  Additionally, for $x<x'$ we have
\begin{align*}
\int |A(x,y;k)- A(x',y;k)| \,dy \leq \|r\|_{L^1} \int_x^{x'}  |q(s)|\,ds,
\end{align*}
which shows that in fact $A:\mathcal C\to\mathcal C$.  Iterating the bounds above shows
$$
\| A^n \|_{\mathcal C\to\mathcal C} \leq  \frac1{n!} \biggl[ \frac{C\bigl(\|q\|_{H^{4,4}}\bigr)}{\langle k\rangle}\biggr]^{n} ,
$$
which in turn guarantees that $\phi$ can be constructed as $\phi = \sum_{n=0}^\infty A^n 1$.

This series construction shows that, when viewed as a $\mathcal C$-valued function, $\phi(x;k,q)$ has real-analytic dependence on $q\in H^{4,4}(\R)$, is a  continuous function of $k$ in the closed upper half-plane, and is a holomorphic function of $k\in \C^+$.  Our estimates on $A$ also show the quantitative bound
$$
\bigl|e^{\frac i2\|q\|_{L^2}^2} \ta(\sqrt k;q) - 1 \bigr| + \sup_x |\phi(x;k,q) - 1| \lesssim \langle k \rangle^{-1},
$$
uniformly for $k\in\overline{ \C^+}$ and $q$ in bounded subsets of $H^{4,4}(\R)$.

Mimicking the arguments above, we first see that
$$
\bigl|\tfrac{d\ }{dk}A(x,y;k)\bigr| \lesssim \langle k\rangle^{-1} C\bigl( \|q\|_{H^{4,4}} \bigr) \langle y \rangle |r(y)|
$$
and then that the $k$-derivative of $\ta(\sqrt k;q)$ is bounded and continuous on the closed upper half-plane.   This does not quite suffice to prove the $\langle k\rangle^{-2}$ decay we require; the sole obstruction is the term
\begin{align*}
\lim_{x\to\infty} \tfrac{d}{dk} [A 1](x) = \int_{-\infty}^\infty\int_y^\infty 2i(s-y) e^{2ik(s-y)+im(s)-im(y)}q(s) r(y)\,ds\,dy,
\end{align*}
which the preceding arguments only show to be $O(\langle k\rangle^{-1})$.  The key to handling this term is to first integrate by parts in $y$:
\begin{align*}
2ik \lim_{x\to\infty} \tfrac{d}{dk} [A 1](x)
	&= - \int_{-\infty}^\infty\int_y^\infty 2i(s-y) \bigl[\tfrac{\partial\ }{\partial y} e^{2ik(s-y)}\bigr] e^{im(s)-im(y)}q(s) r(y)\,ds\,dy \\
&= \int_{-\infty}^\infty\int_{-\infty}^s e^{2ik(s-y)} \tfrac{\partial\ }{\partial y} \bigl[2i(s-y) e^{im(s)-im(y)}q(s) r(y) \bigr] \,dy\,ds
\end{align*}
and only then integrate by parts in $s$.  In this way we obtain
\begin{align*}
\bigl| (2ik)^2 \lim_{x\to\infty} \tfrac{d}{dk} [A 1](x) \bigr|
	&\lesssim  \iint \Bigl| \tfrac{\partial^2\ }{\partial s\partial y} \bigl[2i(s-y) e^{im(s)-im(y)}q(s) r(y) \bigr] \Bigr|\,dy\,ds \\
	&\qquad	+ \int |q(s)r(s)|\,ds,
\end{align*}
which is easily estimated in terms of $\| q \|_{H^{4,4}}$.
\end{proof}

We are now ready to prove the Jost--Pais identity stated earlier as \eqref{E:JP I}.  Our argument is modeled on \cite[Lemma~2.8]{MR2310217}, which retains much of the spirit of the original \cite{MR0044404}.  For a very different approach to such results, see \cite[Proposition~5.7]{MR2154153}.

\begin{lem}[A Jost--Pais identity]\label{l:JP}
For \(q\in \Schwartz\) and \(k\in \C^+\) we have
\begin{equation}\label{JP}
a(k;q) = \ta(\sqrt k;q).
\end{equation}
\end{lem}

\begin{proof} Recall that $a(k;q)$ is a continuous function of $q\in L^2$ and a holomorphic function of $k\in\C^+$.  Similarly, by Lemma~\ref{l:Jost-AUX}, $\ta(\sqrt k;q)$ depends continuously on $q\in H^{4,4}(\R)$ and is also holomorphic for $k\in\C^+$.  Thus, it suffices to prove the identity \eqref{JP} only for \(q\in \Test(\R)\) and at those $k$ where $a(k;q)\neq 0$.

When $a(k;q)\neq 0$, the identity \eqref{E:alt a defn} shows that the linear equation \eqref{2 JP} is uniquely solvable.  We wish to write the solution via Fredholm expansion.  To this end, we first introduce the kernel
\begin{align}
A(x,y) = -k\int^{\infty}_{x\vee y} e^{-ik(x-2s+y)} q(s) \bar  q(y)\,ds
\end{align}
of the operator $A = - k  (-ik - \p)^{-1}q (-ik + \p)^{-1}\bar q$, as well as 
\[
A\begin{pmatrix}x_1,\dots,x_n\\y_1,\dots,y_n\end{pmatrix} =\det\begin{bmatrix}A(x_i,y_j)\end{bmatrix}_{1\leq i,j\leq n}.
\]

As discussed in \cite[Ch. 5]{MR2154153}, the determinant admits the expansion
\[
a(k;q) = \det(1+A) =  1 + \sum_{\ell=1}^\infty\int_{y_1<\dots<y_\ell}A\begin{pmatrix}y_1,\dots, y_\ell\\y_1,\dots,y_\ell\end{pmatrix}\,dy_1\cdots d y_\ell,
\]
and (when this is non-zero) we may write $(1+A)^{-1}=1-B$, where $B$ has kernel
\[
B(x,y) = \frac1{a(k;q)}\sum_{\ell=0}^\infty\int_{y_1<\dots<y_\ell}A\begin{pmatrix}x,y_1,\dots, y_\ell\\y,y_1,\dots,y_\ell\end{pmatrix}\,dy_1\cdots d y_\ell.
\]
Thus, we may write the solution of \eqref{2 JP} as
\begin{align}\label{Psi1B}
\Psi_{11}^-(x) = \ta(k;q)e^{-ikx} \biggl[ 1 - e^{ikx} \int  B(x,y) e^{-iky}\,dy\biggr].
\end{align}

Observe that if $x\leq y_1 \leq \min\{y,y_2,\ldots,y_\ell\}$, then 
\[
A\begin{pmatrix}x,y_1,\dots, y_\ell\\y,y_1,\dots,y_\ell\end{pmatrix} = 0,
\]
because it is the determinant of a matrix whose first two rows are linearly dependent.  On the other hand, if $x\leq y\leq \min\{y_1,y_2,\ldots,y_\ell\}$, then
\[
e^{ikx}  A\begin{pmatrix}x,y_1,\dots, y_\ell\\y,y_1,\dots,y_\ell\end{pmatrix} e^{-iky} = A\begin{pmatrix}y,y_1,\dots, y_\ell\\y,y_1,\dots,y_\ell\end{pmatrix}.
\]
Combining these observations, we find that for $x < \min\supp(q)$,
\begin{align*}
e^{ikx} \!\int \!  B(x,y) e^{-iky}\,dy &= \frac1{a(k;q)}\sum_{\ell=0}^\infty\int_{y<y_1<\dots<y_\ell}A\begin{pmatrix}y,y_1,\dots, y_\ell\\y,y_1,\dots,y_\ell\end{pmatrix}\,dy_1\cdots d y_\ell\,dy\\
&= \frac1{a(k;q)}\bigl( a(k;q) - 1\bigr).
\end{align*}

Returning to \eqref{Psi1B} and recalling that $\Psi_{11}^-(x) = e^{-ikx}$ for $x$ to the left of the support of $q$, we deduce that
\[
1 = \ta \bigl[1 - \tfrac1a(a-1)\bigr],
\]
and hence \(a(k) = \ta(\sqrt k)\), as required.
\end{proof}

Having proved \eqref{JP}, we no longer need these two separate notations.  We favor $a(k;q)$, even when $k\in \R$, despite the fact that all our analysis of these boundary values rests on the scattering interpretation. 

\begin{prop}\label{p:FACTS}
\textup{(i)} The map \(H^{4,4}\ni q\mapsto \bigl(a - e^{-\frac i2\|q\|_{L^2}^2}\bigr)\in X\) is locally Lipschitz.\\
\textup{(ii)} If \(q\in \Schwartz\) and $\theta\in (0, \pi]$, then $a(k;q)$ has only finitely many zeros in the sector $\bigl\{k\in \C: \theta\leq\arg k\leq\pi\bigr\}$.\\
\textup{(iii)} For \(q\in \Schwartz\) and $k\in\R$ we have the following behavior for $a(k;q):$
\begin{equation}\label{Real-Line}
a(0) = 1,\qtq{}|a(k)|\geq 1\text{ if }k<0,\qtq{and}|a(k)|\leq 1\text{ if }k>0.
\end{equation}
\end{prop}
\begin{proof}
The claims follow from Lemmas~\ref{l:Jost-AUX} and~\ref{l:JP}, where we note that \eqref{Real-Line} is a consequence of \eqref{S structure}.
\end{proof}

Our last results in this section concern the zeros of the perturbation determinant. Given $q\in L^2$, \eqref{E:alt a defn} shows that \(a(k;q)\) has a zero at $k=\zero\in \C^+$ if and only if there is a non-zero \(\phi\in L^2\) so that
\begin{equation}\label{EV}
\phi = z (-iz - \p)^{-1}q (-iz + \p)^{-1}\bar q \phi.
\end{equation}

Recalling the definition of $B_M$ from \eqref{BM}, we then have the following:

\begin{lem}\label{l:EF}
Let \(q\in B_M\) and suppose that \(\phi\in L^2\) is a non-zero solution to \eqref{EV} with \(\zero\in \C^+\). Then the vector-valued function $\psi$ with
\[
\psi_1 = \phi  \qtq{and}\psi_2 = -\sqrt \zero\bigl(-i\zero + \p\bigr)^{-1}\bar q\phi
\]
belongs to $H^1(\R)$ and is a solution to
\begin{equation}\label{e:psi}
\psi' = -i\sigma_3\zero\psi + \sqrt \zero\begin{bmatrix}0&q\\-\bar q&0\end{bmatrix}\psi.
\end{equation}

Further, we have the estimates
\begin{align}
\|\psi'\|_{L^2} &\lesssim|\zero|\<M\>\|\psi\|_{L^2},\label{zeta-1}\\
\|\psi\|_{L^4} &\lesssim \bigl(|\zero|\<M\>\bigr)^{\frac14}\|\psi\|_{L^2}.\label{zeta-2}
\end{align}
\end{lem}
\begin{proof}
As \(\phi\in L^2\) is a solution of \eqref{EV} and $q\in L^2$, it is clear that $\phi\in H^1$ and thence that $\psi\in H^1$ and solves \eqref{e:psi}.

The Gagliardo--Nirenberg inequality gives us
\[
\|\psi'\|_{L^2}\lesssim|\zero|\|\psi\|_{L^2} + \sqrt{|\zero|}\|q\|_{L^2}\|\psi\|_{L^\infty}\lesssim|\zero|\|\psi\|_{L^2} + \sqrt{|\zero|}\|q\|_{L^2}\|\psi\|_{L^2}^{\frac12}\|\psi'\|_{L^2}^{\frac12},
\]
and the estimate \eqref{zeta-1} follows from Young's inequality.

Combining the Gagliardo--Nirenberg inequality
\[
\|\psi\|_{L^4}^2\lesssim \|\psi\|_{L^2}^{\frac32}\|\psi'\|_{L^2}^{\frac12},
\]
with \eqref{zeta-1} yields \eqref{zeta-2}.
\end{proof}

\begin{rem}\label{R}
For each $z\in \C^+$, constancy of the Wronskian guarantees that the ODE \eqref{e:psi} admits at most one solution in $H^1$, up to scalar multiples.  By uniqueness for ODEs, this solution is non-vanishing at each $x\in \R$.   Note that this constrains the geometric multiplicity of eigenvalues, but not their algebraic multiplicity.  Thus, we cannot conclude that the perturbation determinant has only simple zeros.
\end{rem}

We now apply Lemma~\ref{l:EF} to obtain:

\begin{prop}[Zero-free region]\label{p:Zeros}
Let \(Q\subseteq B_M\) be equicontinuous. Given \(\epsilon>0\), choose \(N\in 2^\Z\) so that \eqref{HF-UNI} holds. If \(q\in Q\), $z\in \C^+$, and \(a(z;q)=0\), then
\begin{equation}\label{No high zeros}
|\zero|\leq \tfrac N{\epsilon^2}\qtq{or} \arg \zero \lesssim_M \epsilon .
\end{equation}

\end{prop}
\begin{proof}
Applying Lemma~\ref{l:EF}, we may find a solution \(\psi\in H^1\) of \eqref{e:psi} so that \(\|\psi\|_{L^2}=1\). Taking the inner product of \eqref{e:psi} with \(\sigma_3\psi\) and then taking the real part, we obtain
\[
\Im\sqrt \zero = -\Re\<\psi_1,q\psi_2\>.
\]
Applying \eqref{HF} and \eqref{zeta-2}, we then have
\[
\Im\sqrt \zero\leq |\<\psi_1,q\psi_2\>|\lesssim\|q_{\leq N}\|_{L^\infty}\|\psi\|_{L^2}^2 + \|q_{>N}\|_{L^2}\|\psi\|_{L^4}^2\lesssim\sqrt{NM}+ \epsilon\sqrt{|\zero|M\<M\>},
\]
which gives us \eqref{No high zeros}.
\end{proof}

\section{The trace formula}\label{S:4}

An important ingredient in our proof of Theorem~\ref{t:Main} will be the following proposition, which appears as~\cite[Lemma~2.3]{BP}. For completeness, we give a simple, self-contained proof here:

\begin{prop}[Mass $=$ Zeros $+$ Winding]\label{p:WindingAlternative}
Let \(q\in\Schwartz\) and \(\theta\in(0,\pi]\) so that \(a(k;q)\) has no zeros on the ray \(e^{i\theta}\R^+\). Then 
\begin{equation}\label{WindingAlternative}
\|q\|_{L^2}^2 = 4\pi \ell - \tfrac2i\int_0^{+\infty e^{i\theta}}\tfrac{a'(k)}{a(k)}\,dk,
\end{equation}
where \(\ell\) denotes the number of zeros in the region \(\{k\in \C: \theta<\arg k<\pi\}\), counted according to multiplicity.
\end{prop}

\begin{proof} We first treat $\theta=\pi$.  By Proposition~\ref{p:FACTS}, 
\begin{equation}\label{RealWinding}
\int_{-\infty}^0 \tfrac{a'(k)}{a(k)}\,dk - \tfrac i2\|q\|_{L^2}^2 = 2\pi i d(q),
\end{equation}
for some \(d(q)\in \Z\). As the left hand side is continuous as a map from \(\Schwartz\to \C\) and \(d(0) = 0\), we must have \(d(q) = 0\) for all \(q\in \Schwartz\).

For $\theta\in(0,\pi)$, \eqref{WindingAlternative} follows from Proposition~\ref{p:FACTS} and the argument~principle.
\end{proof}

Proposition~\ref{p:WindingAlternative} provides a direct connection between $q$ and the analytic function $a(k;q)$.  The trace formula \eqref{Trace} below has the same structure.  We prove it here for general $q\in\Schwartz$, following the path established already in \cite{MR0303132}.  The trace formula for those $q\in\Schwartz$ without spectral singularities played an important role in \cite{BP}.

\begin{prop}[The trace formula]\label{p:Tracey}
Fix \(q\in \Schwartz\) and let $\{\zero_j\}$ enumerate the zeros of \(a(k;q)\) in $\C^+$, repeated according to multiplicity.  Then there is a finite positive measure $\mu_q$ on $\R$ so that
\begin{equation}\label{Tracey}
a(k) = \prod \tfrac{\bar \zero_j}{\zero_j}\tfrac{k - \zero_j}{k - \bar \zero_j}\exp\biggl[\tfrac1{i\pi} \int_\R \tfrac k{k-s} \,d\mu_q(s)\biggr] \quad\text{for all $k\in \C^+$},
\end{equation}
and the following trace formula holds
\begin{equation}\label{Trace}
\|q\|_{L^2}^2 = 4\sum \arg(\zero_j) + \tfrac2\pi\int_\R d\mu_q(s).
\end{equation}
\end{prop}

\begin{proof} We begin by proving \eqref{Tracey}.  This will follow by combining the properties of $a(k;q)$ described in Proposition~\ref{p:FACTS} with the general theory of inner/outer factorization as discussed, for example, in \cite{MR565451}. 

As $a(k;q)$ is a non-zero bounded holomorphic function in $\C^+$ that extends continuously to the boundary,
\begin{align}\label{Nevan}
\sum \tfrac{\Im z_j}{1+|z_j|^2} < \infty \qtq{and}  \tfrac{1}{1+s^2}\log| a(s;q) | \in L^1(\R).
\end{align}
Moreover, $a(k;q)$ admits the factorization
\begin{align}\label{innerouter}
a(k;q)= e^{i\tilde \varphi} \tilde B(k) \tilde {\mathcal O}(k) \tilde S(k) 
\end{align}
where $\tilde\varphi\in[0,2\pi)$, $\tilde B(k)$ is the Blaschke product formed from the zeros $\{z_j\}$, and the outer factor $\tilde  {\mathcal O}(k)$ and the singular inner factor $\tilde S(k)$ take the form
\begin{align*}
\tilde  {\mathcal O}(k)&= \exp\biggl\{\tfrac{i}{\pi} \int_\R \bigl[\tfrac 1{k-s} + \tfrac{s}{1+s^2}\bigr] \log |a(s)|\,ds\biggr\},\\
\tilde S(k)&=\exp\biggl\{ i \beta k  - \tfrac{i}{\pi} \int_\R \bigl[\tfrac 1{k-s} + \tfrac{s}{1+s^2}\bigr] \,d\sigma(s) \biggr\}.
\end{align*}
Here $\beta\geq 0$ and $\sigma$ is a positive measure satisfying $\int \frac{d\sigma(s)}{1+s^2} <\infty$ that is singular with respect to Lebesgue measure.

By Proposition~\ref{p:FACTS}, we know that there is a $\delta>0$ so that $\delta < |z_j| < \delta^{-1}$ and $0< \arg z_j < \pi -\delta$, uniformly in $j$.  Likewise, it shows that $\log| a(s;q) |$ is $O( |s| )$ as $s\to 0$ and $O(|s|^{-1})$ as $|s|\to\infty$.  In this way, we may upgrade \eqref{Nevan} to
\begin{align}\label{Nevan'}
\sum \arg z_j  < \infty \qtq{and}  0 \leq  - \tfrac{1}{s}\log |a(s;q)| \in L^1(\R),
\end{align}
where positivity is deduced from \eqref{Real-Line}.  These observations allow us to replace $\tilde B(k)$ and $\tilde  {\mathcal O}(k)$ by
\begin{align}\label{BOfac}
B(k) = \prod \tfrac{\bar z_j}{z_j}\tfrac{k - z_j}{k - \bar z_j}    \qtq{and}  
	\mathcal O (k) = \exp\biggl\{\tfrac{i}{\pi} \int_\R \bigl[\tfrac 1{k-s} + \tfrac1s\bigr] \log |a(s)|\,ds\biggr\},
\end{align}
provided we update $\tilde \varphi$ accordingly.

A universal property of the factorization \eqref{innerouter} is that the non-tangential limit of $a(k;q)$ must vanish $\sigma$-almost everywhere.  Consequently, Proposition~\ref{p:FACTS} confines the support of $\sigma$ to a compact subset of $(0,\infty)$.  In particular, $\int\frac{d\sigma(s)}{|s|}<\infty$.  Likewise, as $|a(i\kappa;q)|\to 1$ as $\kappa\to\infty$, we see that $\beta=0$.

In light of all this, we may write $a(k)=e^{i\varphi}B(k){\mathcal O}(k)S(k)$ with singular factor
\begin{align}\label{SIfac}
S(k) = \exp\biggl\{ - \tfrac{i}{\pi} \int_\R \bigl[\tfrac 1{k-s} + \tfrac{1}{s}\bigr] \,d\sigma(s) \biggr\}
\end{align}
and consequently the representation \eqref{Tracey} will hold with the finite positive measure
\begin{align*}
d\mu_q(s)  = - \tfrac{1}{s}\log |a(s;q)|\,ds + \tfrac{1}{s}d\sigma(s),
\end{align*}
once we verify that $e^{i\varphi}=1$.  This last step is easily deduced by setting $k=i\kappa$ and sending $\kappa\to0$; indeed, by the dominated convergence theorem, $\mathcal O(i\kappa)$ and $S(i\kappa)$ converge to 1 as $\kappa\to 0$.

The trace formula \eqref{Trace} now follows from \eqref{WindingAlternative} and \eqref{Tracey}.
\end{proof}

\begin{lem}\label{4.3} Fix \(q\in \Schwartz\).  If the perturbation determinant $a(k;q)$ has no zeros in the sector 
\begin{equation}\label{Sigma}
\Sigma = \bigl\{k\in \C: \tfrac\pi4<\arg k<\pi\bigr\},
\end{equation}
then it admits the lower bound
\begin{equation}\label{LB}
|a(i\kappa)|\gtrsim_M1 \quad\text{uniformly for $\kappa>0$}. 
\end{equation}
\end{lem}

\begin{proof}
Proposition~\ref{p:Tracey} ensures that
\[
0\leq \tfrac2\pi\int_\R d\mu_q(s)\leq M.
\]
Thus, for \(k = i\kappa\) we may bound
\[
\biggl|\exp\biggl[\tfrac1{i\pi} \int_\R \tfrac {i\kappa}{i\kappa-s} \,d\mu_q(s)\biggr] \biggr|\geq \exp\bigl(-\tfrac12M\bigr).
\]

For fixed \(\theta\in(0,\frac\pi4]\), elementary calculus shows that the function
\[
x\mapsto \Bigl| \frac{ix - e^{i\theta}}{ix - e^{-i\theta}}\Bigr|^2
\]
achieves its minimum value on \(\R^+\) at \(x=1\). Thus for \(\kappa,r>0\), 
\[
\log \Bigl|\frac{i\kappa - re^{i\theta}}{i\kappa - re^{-i\theta}}\Bigr|\geq \log \frac{1 - \sin\theta}{1 + \sin\theta}.
\]
Noting that the right-hand side is a concave function of $\theta\in [0, \tfrac\pi 4]$, the Mean Value Theorem guarantees that
\begin{equation}\label{Calculus}
\log \Bigl|\frac{i\kappa - re^{i\theta}}{i\kappa - re^{-i\theta}}\Bigr|\geq -2\sqrt 2\theta \qtq{for} \theta\in [0, \tfrac\pi 4].
\end{equation}

We now enumerate the zeros of \(a(k;q)\) in \(\C^+\) by \(\{z_j\}\). As \(\arg z_j\leq \frac\pi4\), we may apply \eqref{Calculus} to the Blaschke product in \eqref{BOfac} and then use \eqref{Trace} to estimate
\[
|B(i\kappa)| \geq \exp\Bigl(-2\sqrt 2\sum\arg z_j\Bigr)\geq \exp\bigl(-\tfrac1{\sqrt 2}M\bigr).
\]
This completes the proof of \eqref{LB}.
\end{proof}

\section{The base case}\label{S:5}

Our proof of Theorem~\ref{t:Main} will proceed by induction on the number of zeros of the perturbation determinant in the region \(\Sigma\) defined in \eqref{Sigma}. To this end, we make the following definition:

\begin{defn}\label{d:Order}We say that a set \(Q\subseteq \Schwartz\) has order \( J\geq 0\) if, for every \(q\in Q\), \(a(k;q)\) has exactly \( J\) zeros (counted by multiplicity) in the region \(\Sigma\).
\end{defn}

Given a set \(Q\subseteq \Schwartz\), we define the set
\begin{equation}\label{Qs}
Q_* = \bigl\{q\in \Schwartz: a(k;q)\equiv a(k;\tilde q) \text{ for some $\tilde q\in Q$}\bigr\}.
\end{equation}
Here, by $a(k;q)\equiv a(k;\tilde q)$ we mean $a(k;q)= a(k;\tilde q)$ for all $k\in \overline{\C^+}$.
It is clear that if \(Q\) has order \(J\geq 0\), in the sense of Definition~\ref{d:Order}, then \(Q_*\) also has order \(J\). Moreover, taking $\theta=\pi$ in  Proposition~\ref{p:WindingAlternative}, we obtain the following:

\begin{lem}\label{l:COM} If \(q,\tilde q\in \Schwartz\) satisfy \(a(k;q) \equiv a(k;\tilde q)\), then
\[
\|q\|_{L^2} = \|\tilde q\|_{L^2}.
\]
As a consequence, if \(Q\subseteq B_M\) then \(Q_*\subseteq B_M\).
\end{lem}

We are now in a position to prove the base case of our inductive argument:

\begin{prop}[The base case]\label{p:Base}
If \(Q\subseteq B_M\cap \Schwartz\) is an equicontinuous set of order \(0\) then \(Q_*\subseteq B_M\cap \Schwartz\) is equicontinuous.
\end{prop}
\begin{proof}
As \(Q\) is equicontinuous, given a dyadic \(0<\epsilon\leq\frac14\) sufficiently small (depending only on \(M\)), we may find \(N_0\in 2^\Z\) so that \(Q\) satisfies \eqref{HF-UNI}.

From Corollary~\ref{c:NHWU}, the definition \eqref{Qs} of \(Q_*\), and Lemma~\ref{l:COM}, we may bound
\begin{equation}\label{STEP 1a}
\sup_{q\in Q_*}\Bigl|a(i\kappa;q) - e^{-\frac i2\|q\|_{L^2}^2}\Bigr| = \sup_{q\in Q}\Bigl|a(i\kappa;q) - e^{-\frac i2\|q\|_{L^2}^2}\Bigr|\lesssim_M \epsilon^2 \quad\text{for all \(\kappa\geq \tfrac {N_0}{\epsilon^2}\).}
\end{equation}

To prove the proposition, we will show that there exists \(N_* = N_*(\epsilon,N_0,M)\in 2^\Z\) so that
\begin{equation}\label{Target}
\sup_{q\in Q_*}\|q_{>N_*}\|_{L^2}^2\lesssim_M\epsilon^2\log\bigl(\tfrac1\epsilon\bigr).
\end{equation}

For each \(q\in Q_*\), the pigeonhole principle ensures that there exists \(N_1 = N_1(q,\epsilon)\in 2^\Z\) satisfying
\begin{equation}\label{Pigeon}
\tfrac {N_0}{\epsilon^4}\leq N_1\leq \tfrac {N_0}{\epsilon^4}\epsilon^{-6\epsilon^{-2}}
\end{equation}
so that \(q\) satisfies \eqref{Gap} with \(N = N_1\). We then claim that
\begin{equation}\label{lock}
\Bigl|a(i\kappa;q) - a(i\kappa;q_{\leq N_1})\Bigr|\lesssim_M\epsilon^2\log\bigl(\tfrac1\epsilon\bigr) \quad\text{uniformly for $\kappa>0$}.
\end{equation}

As a first step toward proving \eqref{lock}, we employ each part of Proposition~\ref{p:Winding}.  Applying part (i) to \(q_{\leq N_1}\) with \(N =\epsilon^3 N_1\) yields
\begin{equation}\label{No low change}
\Bigl|a(i\kappa;q_{\leq N_1}) - e^{-\frac i2\|q_{\leq N_1}\|_{L^2}^2}\Bigr|\lesssim_M \epsilon^2\qtq{uniformly for}\kappa\geq\epsilon N_1,
\end{equation}
while applying part (ii) to \(q_{>N_1}\) with \(N = \frac{N_1}{\epsilon^3}\) yields
\begin{equation}\label{No high change}
\Bigl|a(i\kappa;q_{> N_1}) - 1\Bigr|\lesssim_M \epsilon^2\log\bigl(\tfrac1\epsilon\bigr)\qtq{uniformly for}0<\kappa\leq\tfrac{N_1}\epsilon.
\end{equation}

Let us first prove \eqref{lock} in the regime \(\kappa\geq\epsilon N_1\):  We estimate
\begin{align*}
\Bigl|a(i\kappa;q) - a(i\kappa;q_{\leq N_1})\Bigr|&\leq \Bigl|a(i\kappa;q) - a(iN_1;q)\Bigr| + \Bigl|a(i\kappa;q_{\leq N_1}) - a(iN_1;q_{\leq N_1})\Bigr| \\
&\qquad + \Bigl|a(iN_1;q) - a(iN_1;q_{\leq N_1})a(iN_1;q_{>N_1})\Bigr|\\
&\qquad + \Bigl|a(iN_1;q_{\leq N_1})\Bigr|\Bigl|a(iN_1;q_{>N_1}) - 1\Bigr|\\
&\lesssim_M \epsilon^2\log\bigl(\tfrac1\epsilon\bigr)
\end{align*}
by using \eqref{STEP 1a} and \eqref{No low change} on the first line, \eqref{Factor} on the second, and \eqref{No high change} on the third. 

To complete the proof of \eqref{lock}, we argue as follows:  If \(0<\kappa\leq\frac{N_1}\epsilon\), then \eqref{Factor} and \eqref{No high change} guarantee that
\begin{align*}
\Bigl|a(i\kappa;q) - a(i\kappa;q_{\leq N_1})\Bigr|&\leq\Bigl|a(i\kappa;q) - a(i\kappa;q_{\leq N_1})a(i\kappa;q_{>N_1})\Bigr|\\
&\quad  + \Bigl|a(i\kappa;q_{\leq N_1})\Bigr|\Bigl|a(i\kappa;q_{>N_1}) - 1\Bigr|\\
&\lesssim_M \epsilon^2\log\bigl(\tfrac1\epsilon\bigr).
\end{align*}

By assumption, \(a(k;q)\) has no zeros in \(\Sigma\); thus \eqref{LB} holds.  Combining this with \eqref{lock} shows that for $\eps$ sufficiently small (depending only on $M$),
\[
\bigl|a(i\kappa;q)\bigr|\gtrsim_M1 \qtq{and} \bigl|a(i\kappa;q_{\leq N})\bigr|\gtrsim_M1.
\]

Combining this with \eqref{lock} yields (mirroring Rouch\'e's Theorem) 
\[
 \limsup_{\kappa\to\infty} \biggl|\int_0^\kappa \tfrac{a'(i\vk;q)}{a(i\vk;q)}\,d\vk - \int_0^\kappa \tfrac{a'(i\vk;q_{\leq N})}{a(i\vk;q_{\leq N})}\,d\vk\biggr|\lesssim_M\epsilon^2\log\bigl(\tfrac1\epsilon\bigr),
\]
provided $\eps$ is sufficiently small (depending only on $M$).  Thus, using \eqref{WindingAlternative} we deduce
\begin{align*}
\|q_{>4N_1}\|_{L^2}^2
	\leq \|q\|_{L^2}^2 - \|q_{\leq N_1}\|_{L^2}^2
	\lesssim_M \epsilon^2\log\bigl(\tfrac1\epsilon\bigr).
\end{align*}
This proves that \eqref{Target} holds with the choice \(N_* = 4N_0\epsilon^{-4}e^{-6\epsilon^{-2}}\).
\end{proof}

\section{The B\"acklund transform}\label{S:6}

We begin this section by describing the B\"acklund transform for \eqref{DNLS}, closely following~\cite{MR3702542}.  The ultimate purpose of this section is to determine the behavior of equicontinuous sets under this transformation; this is the subject of Proposition~\ref{p:Und}.

If \(q\in \Schwartz\) and \(a(k;q)\) has a zero at \(k=\zero\in \C^+\), Lemma~\ref{l:EF} ensures that we can find a non-zero solution \(\psi = (\psi_1,\psi_2)^T\in H^1\) of \eqref{e:psi}.  In view of Remark~\ref{R}, $\psi$ is unique up to scalar multiples and nowhere vanishing.  We then define
\begin{equation}\label{Backlund-IP}
\bd(x;z) = \sqrt\zero |\psi_1(x)|^2 + \sqrt{\bar \zero} |\psi_2(x)|^2
\end{equation}
and observe that, since $z\in \C^+$, 
\begin{equation}\label{d-lower}
|\bd(x;z)| \geq |\Re\sqrt\zero| |\psi(x)|^2 >0.
\end{equation}
As a consequence, we may define
\begin{equation}\label{BT-parts}
\bG(x;z) = \frac{\bar\bd(x;z)}{\bd(x;z)}\qtq{and}\bS(x;z) = 4\Im\zero\,\frac{\psi_1(x)\bar\psi_2(x)}{\bd(x;z)}.
\end{equation}
The corresponding B\"acklund transform of \(q\) is then defined to be
\begin{equation}\label{Backlund}
\bB(q;z) = \bG^2q + \bG\bS.
\end{equation}
Note that $\bG$, $\bS$, and $\bB$ depend on $z$ and $q$, but not on $\psi$, due to homogeneity and Remark~\ref{R}.

\begin{lem}\label{l:Backlund}
Let \(q\in \Schwartz\) be such that \(a(k;q)\) has a zero at \(k = \zero\in \C^+\). Then $\bG\in C^\infty$ and $\bS\in \Schwartz$ satisfy
\begin{alignat}{5}
\|\bG\|_{L^\infty} &= 1&\qtq{and}&&\|\bG'\|_{L^2}&\lesssim |\Im \sqrt\zero|\Bigl(\|\bB(q)\|_{L^2} + \|q\|_{L^2}\Bigr),\label{Gmu}
\\
\|\bS\|_{L^\infty}&\leq4|\Im\sqrt\zero|&\qtq{and}&&\|\bS'\|_{L^2}&\lesssim |\zero|\Bigl(\|\bB(q)\|_{L^2} + \|q\|_{L^2}\Bigr).\label{Smu}
\end{alignat}

Moreover, \(\bB(q;z)\in \Schwartz\) and for any \(k\in \overline{\C^+}\backslash\{\zero\}\) we have
\begin{equation}\label{RemZeros}
a(k;\bB(q;z)) = \tfrac{\zero}{\bar\zero}\tfrac{k - \bar \zero}{k - \zero}a(k;q).
\end{equation}
\end{lem}

\begin{proof}
Given $\lambda^2\in\C^+$, Levinson's Theorem \cite[\S3.8]{MR0069338} guarantees the existence of solutions $\Psi^\pm(x;\lambda)$ to
\begin{equation}\label{Levin-psi}
\partial_x \Psi^\pm = -i\sigma_3\lambda^2\Psi^\pm + \lambda\begin{bmatrix}0&q\\-\bar q&0\end{bmatrix}\Psi^\pm
\qtq{with} \lim_{x\to\pm\infty} \, \Psi^\pm(x) e^{i\lambda^2x\sigma_3}  = \Id.
\end{equation}

From the boundary conditions, we see that the column $\Psi_1^-(x)$ decays exponentially as $x\to-\infty$, while $\Psi_2^-(x)$ grows exponentially.  Analogously, $\Psi_1^+(x)$ grows and $\Psi_2^+(x)$ decays as $x\to+\infty$.  Note that the decaying Jost solutions are unique; the growing solutions are not (one may add any multiple of the decaying column).  While this ambiguity prevents one from defining $\tb(\lambda;q)$ for $\lambda^2\notin\R$, it does not affect $\ta(\lambda;q)$ since we always have the Wronskian relation
\begin{equation}\label{a from Wron}
\ta(\lambda; q) = \det\begin{bmatrix}\Psi_1^-(x,\lambda) & \Psi_2^+(x,\lambda)\end{bmatrix},
\end{equation}
which only involves the decaying solutions.

Let $\mu=\sqrt z\in\C^+$.  When $\lambda=\mu$, the ODE in \eqref{Levin-psi} is precisely that satisfied by $\psi$.  On the other hand, the columns of $\Psi^-(x;\mu)$ (and likewise those of $\Psi^+(x;\mu)$) form a basis of solutions to this linear problem.  As $\psi$ is non-zero and square integrable, we deduce that
\[
\psi(x) = C_-\Psi_1^-(x;\mu) = C_+\Psi_2^+(x;\mu),
\]
for some non-zero constants \(C_\pm\in \C\). Using the equation \eqref{e:psi}, we may write
\begin{align*}
e^{-i\mu^2x}\psi(x) &= \begin{bmatrix}0\\C_+\end{bmatrix} - \mu\int_x^\infty \begin{bmatrix}0&e^{-2i\mu^2(x-y)}q(y)\\-\bar q(y)&0\end{bmatrix}e^{-i\mu^2y}\psi(y)\,dy,\\
e^{i\mu^2x}\psi(x) &= \begin{bmatrix}C_-\\0\end{bmatrix} + \mu\int_{-\infty}^x \begin{bmatrix}0&q(y)\\-e^{2i\mu^2(x-y)}\bar q(y)&0\end{bmatrix}e^{i\mu^2y}\psi(y)\,dy.
\end{align*}
As \(q\in \Schwartz\), we may use these expressions to show that, for any \(\eta\in C^\infty\) that is supported on \([-1,\infty)\) and identically \(1\) on \([1,\infty)\), we have
\[
\eta(x)\Biggl(e^{-i\mu^2x}\psi(x) - \begin{bmatrix}0\\C_+\end{bmatrix}\Biggr),\bigl[1 - \eta(x)\bigr]\Biggl(e^{i\mu^2x}\psi(x) - \begin{bmatrix}C_-\\0\end{bmatrix}\Biggr)\in\Schwartz.
\]
This suffices to show that \(\bG\in C^\infty\), \(\bS\in \Schwartz\), and
\begin{equation}\label{Glim}
\bG(x) \to \begin{cases}\frac{\mu}{\bar\mu}&\qtq{as}x\to+\infty,\medskip\\\frac{\bar\mu}{\mu}&\qtq{as}x\to-\infty.\end{cases}
\end{equation}

From \eqref{d-lower} and \eqref{BT-parts}, for all \(x\in \R\) we have
\[
|\bG| = 1\qtq{and}|\bS|\leq 4|\Im\sqrt\zero|.
\]
Using \eqref{e:psi}, we compute
\[
\bd' = 2i\Im\zero\Bigl[q\bar \psi_1\psi_2 - i\sqrt\zero|\psi_1|^2 + i\sqrt{\bar\zero}|\psi_2|^2\Bigr] .
\]
Using also the identities 
\begin{align}\label{identities}
\bG^2 - 1 = -\tfrac{2i\Im\zero}{\bd^2}\bigl(|\psi_1|^4 - |\psi_2|^4\bigr) \qtq{and}
|\mu|^2(\bG+ \bar{\bG}) -\tfrac14|\bS|^2= 2\Re(\mu^2),
\end{align}
we obtain
\begin{equation}\label{dGS}
\bG' = \tfrac1{2i}\Bigr[\bar\bS \bB(q) + \bS\bar q\Bigr] \quad\text{and}\quad 
\bS' = \tfrac{2|\zero|}{i}\bigl[\bB(q) - q\bigr].
\end{equation}
The estimates \eqref{Gmu} and \eqref{Smu} follow immediately from these expressions.

Finally, we turn to proving \eqref{RemZeros}. Fixing \(\lambda\in \C\) so that \(\Im(\lambda^2)\geq0\) and \(\lambda^2\neq \mu^2\), we define 
\[
P(x;\lambda) = \frac\mu{\bar \mu}\frac1{\lambda^2 - \mu^2}\begin{bmatrix}\lambda^2\bG - |\mu|^2 & -\frac i2\lambda\bS\bigskip\\ -\frac i2\lambda\bar\bS & \lambda^2\bar\bG - |\mu|^2\end{bmatrix}.
\]
This matrix is invertible; indeed, by the second identity in \eqref{identities},
\begin{equation*}
\det P = \frac{\mu^2}{\bar \mu^2}\frac{\lambda^2 - \bar\mu^2}{\lambda^2 - \mu^2}.
\end{equation*}

Consider now \(\Upsilon^\pm(x;\lambda) := P(x;\lambda)\Psi^\pm(x;\lambda)\).  Applying Remark~\ref{R:2} with \eqref{dGS} and \eqref{Glim} we find that
\[
\p_x \Upsilon^\pm = -i\sigma_3\lambda^2 \Upsilon^\pm + \lambda\begin{bmatrix}0&\bB(q)\\-\bar\bB(q)&0\end{bmatrix}\Upsilon^\pm
	\qtq{with} \lim_{x\to\pm\infty} \, \Upsilon^\pm(x) e^{i\lambda^2x\sigma_3}  = \Id.
\]

Using this and \eqref{a from Wron} we obtain
\begin{equation*}
a(\lambda^2; \bB(q)) = \det\begin{bmatrix}\Upsilon_1^-(\lambda) &\Upsilon_2^+(\lambda)\end{bmatrix}=\det\begin{bmatrix}P\Psi_1^-(\lambda) &P \Psi_2^+(\lambda)\end{bmatrix}=(\det P) a(\lambda^2; q),
\end{equation*}
which proves \eqref{RemZeros}.
\end{proof}

Our next result will be useful in proving that the B\"acklund transform preserves equicontinuity.

\begin{lem}
Let \(q\in \Schwartz\) be such that \(a(k;q)\) has a zero at \(k = \zero\in \C^+\). Then
\begin{equation}\label{MassReduction}
\|\bB(q)\|_{L^2}^2 = \|q\|_{L^2}^2 -4\arg\zero,
\end{equation}
and there exists an absolute constant \(C>0\) so that if \(N\in 2^\Z\),
\begin{align}
\bigl\|P_{>N}[\bB(q)]\bigr\|_{L^2} &\leq \|q_{>\frac N8}\|_{L^2} + C \sqrt{\tfrac{|\zero|}N} \|q\|_{L^2}\Bigl(\sqrt {\tfrac{|\zero|}N} + \|q\|_{L^2}\Bigr),\label{Surgery1}\\
\|q_{>N}\|_{L^2} &\leq \bigl\|P_{>\frac N8}[\bB(q)]\bigr\|_{L^2} + C \sqrt{\tfrac{|\zero|}N} \|q\|_{L^2}\Bigl(\sqrt {\tfrac{|\zero|}N} + \|q\|_{L^2}\Bigr).\label{Surgery2}
\end{align}
\end{lem}
\begin{proof}
A quick computation using the second identity in \eqref{identities} and \eqref{dGS} yields
\begin{align}\label{q Bq}
|q|^2-|\bB(q)|^2 = -2i \bar{\bG} \bG'.
\end{align}
As $z\in \C^+$, formula \eqref{Backlund-IP} shows that $\bd(x;z)$ lies in the open right-half plane, and so $\bG(x; z)\in \C\setminus (-\infty, 0]$.
Thus, integrating \eqref{q Bq} and using \eqref{Glim}, we derive the identity \eqref{MassReduction}.

To prove \eqref{Surgery1}, we decompose
\[
P_{>N}[\bB(q)] = P_{>N}\bigl[\bG^2q_{>\frac N8}\bigr] + P_{>N}\bigl[\bG^2q_{\leq \frac N8}\bigr] + P_{>N}\bigl[\bG\bS\bigr].
\]
For the first term we use \eqref{Gmu} to bound
\[
\bigl\|P_{>N}\bigl[\bG^2 q_{>\frac N8}\bigr]\bigr\|_{L^2}\leq \bigl\|q_{>\frac N8}\bigr\|_{L^2}.
\]
For the second, by first considering the Fourier support and then applying \eqref{Gmu} and \eqref{MassReduction} with Bernstein's inequality, we get
\begin{align*}
\bigl\|P_{>N}\bigl[\bG^2q_{\leq \frac N8}\bigr]\bigr\|_{L^2} &= \bigl\|q_{\leq \frac N8}P_{>\frac N8}\bigl[\bG^2\bigl]\bigr\|_{L^2}\\
&\lesssim N^{-\frac12}\|\bG\|_{L^\infty}\|\bG'\|_{L^2}\|q\|_{L^2}\\
&\lesssim N^{-\frac12}|\zero|^{\frac12}\|q\|_{L^2}^2.
\end{align*}
For the final term we use \eqref{Gmu}, \eqref{Smu}, and \eqref{MassReduction} to bound
\[
\|P_{>N}\bigl[\bG\bS\bigr]\|_{L^2}\lesssim \tfrac1N\Bigl(\|\bG\|_{L^\infty}\|\bS'\|_{L^2} + \|\bG'\|_{L^2}\|\bS\|_{L^\infty}\Bigr)\lesssim\tfrac1N |\zero| \|q\|_{L^2}.
\]

The proof of \eqref{Surgery2} is essentially identical to that of \eqref{Surgery1}, once we have observed that \eqref{Backlund} may be rewritten in the form
\[
q = \bar\bG^2\bB(q) - \bar\bG\bS. \qedhere
\]
\end{proof}

We are now ready to prove that the B\"acklund transform preserves equicontinuity.  Let us first define a set-valued B\"acklund transform.  For a set $Q\subseteq \Schwartz$, we define
\begin{align*}
\cB(Q)=\bigl\{\bB(q;z): q\in Q, \ z\in \Sigma, \text{ and } a(z;q)=0\bigr\},
\end{align*}
where $\Sigma$ is as defined in \eqref{Sigma}.

\begin{prop}\label{p:Und}
Let \(M>0\) and \(Q\subseteq B_M\cap \Schwartz\) be a set of order \( J\geq 1\) in the sense of Definition~\ref{d:Order}. Taking \(Q_*\) to be defined as in \eqref{Qs}, we have the following properties:
\begin{itemize}
\item[(i)] \(\cB(Q)\subseteq \cB(Q_*) \subseteq \cB(Q)_*\) are subsets of \(B_M\cap \Schwartz\) of order \( J-1\).
\item[(ii)] If $Q$ is equicontinuous, then \(\cB(Q)\) is equicontinuous.
\item[(iii)] If $Q$ and \(\cB(Q)_*\) are equicontinuous, then \(Q_*\) is equicontinuous.
\end{itemize}
\end{prop}

\begin{proof}
By construction and Lemma~\ref{l:Backlund}, we have \(\cB(Q)\subseteq \cB(Q_*)\subseteq \cB(Q)_*\subseteq \Schwartz\), where all three sets \(\cB(Q)\), \(\cB(Q_*)\), and \(\cB(Q)_*\) have order \(J-1\). As \(Q\subseteq B_M\), Lemma~\ref{l:COM} and the identity \eqref{MassReduction} yield \(\cB(Q)_*\subseteq B_M\). This completes the proof of~(i).

If \(Q\) is equicontinuous, Proposition~\ref{p:Zeros} ensures that
\begin{equation}\label{Zombie}
\sup\bigl\{|\zero|: q\in Q_*, \ z\in \Sigma, \text{ and } a(z;q)=0\bigr\}\lesssim_{M,Q}1.
\end{equation}
Thus \eqref{Surgery1} gives
\[
\limsup_{N\to\infty}\sup_{q\in \cB(Q)}\|q_{>N}\|_{L^2}\lesssim \limsup_{N\to\infty}\sup_{q\in Q}\|q_{>N}\|_{L^2} = 0,
\]
which proves (ii).

Finally, if \(\cB(Q)_*\) is equicontinuous then \(\cB(Q_*)\subseteq \cB(Q)_*\) is also equicontinuous and  \eqref{Surgery2} and \eqref{Zombie} then give us
\[
\limsup_{N\to\infty}\sup_{q\in Q_*}\|q_{>N}\|_{L^2}\lesssim \limsup_{N\to\infty}\sup_{q\in \cB(Q_*)}\|q_{>N}\|_{L^2} = 0,
\]
which completes the proof of (iii).
\end{proof}

\section{Induction on zeros}\label{S:7}

In this brief section we complete the proof of Theorem~\ref{t:Main}. We proceed by induction on the number of zeros in the sector $\Sigma$ defined in \eqref{Sigma}; this is formalized as the notion of order introduced in Definition~\ref{d:Order}.

\begin{prop}[Induction on zeros]\label{p:Induction}
If \(Q\subseteq B_M\cap \Schwartz\) is an equicontinuous set of order \(J\geq 0\) then \(Q_*\subseteq B_M\cap \Schwartz\), defined as in \eqref{Qs}, is equicontinuous.
\end{prop}
\begin{proof}
We proceed by induction on \(J\). Proposition~\ref{p:Base} provides the base case $J=0$. For the inductive step, we fix $J\geq 1$ and suppose that the statement is true for sets of order \(J-1\).

If $Q\subseteq B_M\cap \Schwartz$ is an equicontinuous set of order $J$, Proposition~\ref{p:Und}(ii) ensures that \(\cB(Q)\subseteq B_M\cap \Schwartz\) is an equicontinuous set of order \(J-1\). Applying the inductive hypothesis to \(\cB(Q)\) we conclude that \(\cB(Q)_*\) is equicontinuous. Proposition~\ref{p:Und}(iii) then yields that \(Q_*\) is equicontinuous, as required.
\end{proof}

We finish by showing that our main result, Theorem~\ref{t:Main}, follows readily from this proposition: 

\begin{proof}[Proof of Theorem~\ref{t:Main}]
For each \(q\in Q\), Proposition~\ref{p:FACTS} ensures that \(a(k;q)\) has at most a finite number of zeros in \(\Sigma\). Moreover, \eqref{Trace} ensures there are at most \(\lfloor\frac{M}{\pi}\rfloor\) such zeros. As a consequence, we can decompose
\[
Q = \bigcup_{J=0}^{\lfloor\frac{M}{\pi}\rfloor} Q\sbrack J,
\]
where \(Q\sbrack J\subseteq B_M\cap \Schwartz\) is a set of order \(J\). As $Q$ is equicontinuous, each $Q\sbrack J$ is equicontinuous, as well.

By Proposition~\ref{p:Induction}, each \(Q\sbrack J_*\subseteq B_M\cap \Schwartz\) is equicontinuous, and hence
\[
Q_* = \bigcup_{J=0}^{\lfloor\frac{M}{\pi}\rfloor} Q\sbrack J_*
\]
is equicontinuous.
\end{proof}

\bibliographystyle{habbrv}
\bibliography{refs}

\end{document}